\documentclass[edts]{article}%
\usepackage{graphicx}
\usepackage{color}
\usepackage{amsmath}
\usepackage{a4}
\usepackage{amsfonts}
\usepackage{amssymb}
\providecommand{\U}[1]{\protect \rule{.1in}{.1in}}
\newtheorem{theorem}{Theorem}

\newtheorem{application}[theorem]{Application}

\newtheorem{corollary}[theorem]{Corollary}

\newtheorem{definition}[theorem]{Definition}

\newtheorem{lemma}[theorem]{Lemma}

\newtheorem{proposition}[theorem]{Proposition}
\newtheorem{remark}[theorem]{Remark}

\newenvironment{proof}[1][Proof]{\textbf{#1.} }{\  \rule{0.5em}{0.5em}}

\begin{document}

\title{On the geodesic flow on CAT(0) spaces}
\author{Charalampos Charitos, Ioannis Papadoperakis
\and and Georgios Tsapogas\\Agricultural University of Athens}
\maketitle

\begin{abstract}
Under certain assumptions on CAT(0) spaces, we show that the geodesic flow is
topologically mixing. In particular, the Bowen-Margulis' measure finiteness
assumption used in \cite{[Ric]} is removed. We also construct examples of
CAT(0) spaces which do not admit finite Bowen-Margulis measure. \newline%

\end{abstract}

\section{Introduction}

Topological transitivity and topological mixing of the geodesic flow are two
dynamical properties extensively studied for Riemannian manifolds. Anosov in
\cite{[Ano]} first proved topological transitivity of the geodesic flow for
compact manifolds of negative curvature. Eberlein in \cite{[Ebe2]} proved
topological mixing for a large class of manifolds. In particular, he
established topological mixing for complete finite volume manifolds of
negative curvature as well as for compact manifolds of non-positive curvature
not admitting isometric, totally geodesic embedding of $\mathbb{R}^{2}.$ The
latter is the class of the so called visibility manifolds (see \cite{[EON]}
and \cite{[Ebe2]}) and, in modern terminology, it can equivalently be
described as the class of compact $CAT(0)$ manifolds which are hyperbolic in
the sense of Gromov (see \cite[Ch. II, Th. 9.33]{[BrH]}). For certain classes
of quotients of $CAT(-1)$ spaces by discrete groups of isometries, topological
mixing was shown in \cite{[CT3]}. All the above results are along the lines of
Eberlein's approach where the following two properties of the universal
covering were essential:

\begin{description}
\item[(u)] uniqueness of geodesic lines joining two boundary points at
infinity and

\item[(c)] the distance of asymptotic geodesics tends, up to
re-parametrization, to zero. 
\end{description}

Recently R. Ricks (see \cite{[Ric]}) made a significant generalization by
proving mixing of the Bowen--Margulis measure under the geodesic flow on all
rank one CAT(0) spaces under the natural assumption that the Bowen-Margulis
measure (also constructed in \cite{[Ric]} for CAT(0) spaces) is finite. In
this work we extend the classical approach of Eberlein to show topological
mixing of the geodesic flow for a class of spaces $X$ which are quotients of a
CAT(0) space $\widetilde{X}$ by a non-elementary discrete group of isometries
$\Gamma$ such that $\partial \widetilde{X}$ is connected and equal to the limit
set $\Lambda \left(  \Gamma \right)  .$ We impose certain conditions on the
CAT(0) space $\widetilde{X}$ (see Standing Assumptions, after Definition
\ref{stas} below), but we allow the Bowen-Margulis measure to be infinite.
Observe that the action of $\Gamma$ is not assumed to be co-compact.

In \cite[Theorem 1.2]{[Dal]} finite volume $n$-dimensional manifolds ($n\geq
2$) of pinched negative curvature are constructed whose fundamental group is
convergent. The latter implies, in particular, that the Bowen-Margulis measure
is infinite. From these examples, one can easily construct CAT(0) spaces which
are hyperbolic in the sense of Gromov, not admitting finite Bowen--Margulis
measure and satisfy our Standing Assumptions.

\section{Definitions and Preliminaries\label{standing}}

Let $Y$ be a proper metric space.

\begin{definition}
A geodesic segment in $Y$ is an isometric map $h:[a,b]\rightarrow Y.$ If
$x=h(a)$ and $y=h(b)$ then a geodesic segment joining $x$ and $y$ will be
denoted by $[x,y]$ and its interior by $\left(  x,y\right)  .$ \newline Let
$I=[0,+\infty)$ or $I=(-\infty,+\infty).$ A geodesic line (resp. geodesic ray)
in $Y$ is a local isometric map $h:I\rightarrow Y$ where $I=\left(
-\infty,+\infty \right)  $ (resp. $I=\left[  0,+\infty \right)  $ ).\newline A
closed geodesic is a local isometric map $h:I\rightarrow Y$ which is a
periodic map.\newline A metric space is called geodesic if every two points
can be joined by a geodesic segment.\newline A geodesic metric space is called
geodesically complete if every geodesic segment extends to a geodesic line.
\end{definition}

\begin{definition}
We say that the metric space $Y\ $is a Hadamard space if $Y\ $is simply
connected, complete, geodesic and has curvature $\leq0.$
\end{definition}

We refer the reader to \cite{[Bal]} and \cite{[BrH]} for a systematic
treatment of Hadamard spaces.

Throughout this paper, $X$ will denote the quotient space $\widetilde
{X}/\Gamma$ where $\widetilde{X}$ is a Hadamard space and $\Gamma$ a
non-elementary discrete group acting freely by isometries on $\widetilde{X}.$
In Section \ref{tesduo} $X$ will, in addition, be a $2$-dimensional surface.
Denote by $p:\widetilde{X}\rightarrow X$ the covering projection. $\Gamma$ is
isomorphic to $\pi_{1}(X)$ and we will make no distinction between $\Gamma$
and $\pi_{1}(X).$

The visual boundary $\partial \widetilde{X}$ of $\widetilde{X}$ is defined by
means of geodesic rays (see \cite[Ch.2, \S 3 p.21]{[CDP]}). Recall that two
geodesic rays $g_{1},g_{2}$ (or geodesics) in $\widetilde{X}$ are called
\textit{asymptotic} if $d\bigl(g_{1}\left(  t\right)  ,g_{2}\left(  t\right)
\bigr)$ is bounded for all $t\in \mathbb{R}^{+}.$ Equivalently, if $g\left(
+\infty \right)  $ denotes the boundary point determined by the geodesic ray
$g|_{\left[  0,+\infty \right)  },$ two geodesic rays $g_{1},g_{2}$ (or
geodesics) in $\widetilde{X}$ are asymptotic if $g_{1}\left(  +\infty \right)
=g_{2}\left(  +\infty \right)  .$ Since $\widetilde{X}$ is a CAT(0) space,
geodesic lines and geodesic rays in $\widetilde{X}$ are global isometric maps.
Note also that geodesic segments with given endpoints are unique. This is just
the Hadamard Cartan Theorem for CAT(0) spaces (see \cite[Th. 4.5 Ch.I]%
{[Bal]}). Moreover, we have uniqueness of geodesic rays in the following
sense: for any $x\in$ $\widetilde{X},$ $\xi \in \partial \widetilde{X}$ there is
a unique geodesic ray $r:[0,\infty)\rightarrow \widetilde{X}\cup \partial
\widetilde{X}$ such that $r(0)=x,$ $r(\infty)=\xi$ (see \cite[Ch.II, Prop.
8.2]{[BrH]}). The corresponding result for geodesic lines is not true.
However, the following theorem holds (see \cite[Cor. 5.8.ii Ch.I]{[Bal]})

\begin{theorem}
[Flat Strip Theorem]If $f,g:\mathbb{R}\rightarrow \widetilde{X}$ are two
geodesics with $f\left(  \infty \right)  $ $=g\left(  \infty \right)  $ and
$f\left(  - \infty \right)  =g\left(  - \infty \right)  $ then $f$ and $g$ bound
a flat strip, that is, a convex region isometric to the convex hull of two
parallel lines in the flat plane.
\end{theorem}

\begin{definition}
\label{stas}We say that $f:\mathbb{R}\rightarrow \widetilde{X}$ is a unique
geodesic if for any geodesic $g:\mathbb{R}\rightarrow \widetilde{X}$ with
$f\left(  -\infty \right)  =g\left(  -\infty \right)  $ and $f\left(
\infty \right)  =g\left(  \infty \right)  ,$ $g$ is a re-parametrization of $f.$
We say that $f:\mathbb{R}\rightarrow \widetilde{X}$ is a closed (resp.
non-closed) geodesic if $p\left(  f\right)  $ is closed, that is, periodic
(resp. non-closed, that is, not periodic) in $X.$
\end{definition}

The \textit{limit set} $\Lambda(\Gamma)$ of $\Gamma$ is defined to be
$\Lambda(\Gamma)=\overline{\Gamma x}\cap \partial \widetilde{X},$ where $x$ is
an arbitrary point in $\widetilde{X}.$ Since the action of $\Gamma$ on
$\widetilde{X}$ is not assumed to be co-compact, it does not follow in general
that $\Lambda(\Gamma)=\partial \widetilde{X}.$ However, we assume throughout
that $\Lambda(\Gamma)=\partial \widetilde{X}.$

For each non-trivial element $\varphi \in \Gamma$ and each $x\in \widetilde{X}$ the sequence
$\varphi^{n}(x)$ (resp. $\varphi^{-n}(x))$ has a limit point $\varphi
(+\infty)$ (resp. $\varphi(-\infty))$ in $\partial \widetilde{X}$ when
$n\rightarrow+\infty.$ This is equivalent to saying that $\Gamma$ has no
elliptic elements which holds as the action of $\Gamma$ is assumed to be free
(see \cite[Ch.II, Prop. 3.2]{[Bal]}). However, as $\Gamma$ can contain
parabolic elements, $\varphi(+\infty)$ and $\varphi(-\infty)$ may coincide. In
the case $\phi$ is a hyperbolic element of $\Gamma,$ the point $\varphi
(+\infty)$ is called \textit{attractive} and the point $\varphi(-\infty)$
\textit{repulsive} point of $\varphi.$

As $\Gamma$ is a discrete group of isometries of $\widetilde{X}$ we have the
following result from \cite{[Coo]}

\begin{proposition}
[{Prop. 1.7 Chapter II in \cite{[Coo]}}]\label{gammaclear} Let $\varphi$ be a
hyperbolic element of $\Gamma$ and $\psi$ any element of $\Gamma.$ If
$Fix(\psi)$ is the set of points in $\partial \widetilde{X}$ fixed by the
action of $\psi, $ then either
\[
\left \{  \varphi(-\infty) ,\varphi(+\infty) \right \}  \cap Fix(\psi) =
\emptyset \mathrm{\ or,\ }\left \{  \varphi(-\infty) ,\varphi(+\infty) \right \}
\subset Fix(\psi) .
\]

\end{proposition}

It follows that $f,g$ are two closed non-homotopic geodesics then $f$ and $g$
cannot be asymptotic. Thus if $F_{h}\subset \partial \widetilde{X}$ denotes the
set of limit points of all hyperbolic elements of $\Gamma$ then $F_{h}$ splits
as the disjoint union
\[
F_{h}=F_{h}^{u}\sqcup F_{h}^{nu}%
\]
where
\[
F_{h}^{nu}:=\left \{  \xi \in \partial \widetilde{X}\bigm \vert \xi=g\left(
+\infty \right)  \mathrm{\ for\ some\ }g\mathrm{\mathrm{\ }%
\ closed\ and\ non-unique}\right \}
\]
and
\[
F_{h}^{u}:=\left \{  \xi \in \partial \widetilde{X}\bigm \vert \xi=g\left(
+\infty \right)  \mathrm{\ for\ some\ }g\mathrm{\mathrm{\ }%
\ closed\ and\ unique}\right \}
\]
Observe that $F_{h}^{u},F_{h}^{nu}$ are invariant under the action of
$\Gamma.$

\noindent \textbf{Standing Assumptions}: Let $X=\widetilde{X}/\Gamma$ where
$\widetilde{X}$ is a proper and geodesically complete CAT(0) space with $\partial \widetilde{X}$ connected and
$\Gamma$ a non-elementary discrete group of isometries acting freely on
$\widetilde{X}$ with $\Lambda(\Gamma) = \partial \widetilde{X}$ such that
$\widetilde{X}$ satisfies the following conditions:

\begin{description}
\item[($\mathrm{\Delta}$)] the space $\widetilde{X}$ is hyperbolic in the
sense of Gromov.

\item[(U)] if $f$ is a non-closed geodesic in $\widetilde{X}$, then $f$ is unique.

\item[(C)] if $f,g$ are asymptotic geodesics with $f\left(  +\infty \right)
=g\left(  +\infty \right)  \in \partial \widetilde{X}\setminus F_{h}^{nu}$ then
for appropriate parametrizations of $f,g$
\[
\lim_{t\rightarrow \infty}d\bigl(f\left(  t\right)  ,g\left(  t\right)
\bigr)=0.
\]

\item[(D)] The set
\[
\left \{  \left(  g(+\infty),g(-\infty)\right)
:g\mathrm{\ is\ closed\ and\ unique}\right \}
\]
is dense in $\partial^{2} \widetilde{X}.$
\end{description}

The geodesic flow for a complete geodesic metric space $X$ is defined by the map

\begin{center}
$\mathbb{R}\times GX\rightarrow GX$
\end{center}

\noindent where $GX$ is the space of all local isometric maps $g:\mathbb{R}%
\rightarrow X$ (see Section \ref{sectionproperties} below for precise
definition and properties) and the action of $\mathbb{R}$ is given by right
translation, i.e. for all $t\in \mathbb{R}$ and $g\in GX$, $\left(  t,g\right)
\rightarrow t\cdot g$ where $t\cdot g:\mathbb{R}\rightarrow X$ is the geodesic
defined by $\left(  t\cdot g\right)  \left(  s\right)  =g\left(  s+t\right)
,s\in \mathbb{R}.$

\begin{definition}
\label{transitivity}The geodesic flow $\mathbb{R}\times GX\rightarrow GX$ is
topologically transitive if given any non-empty open sets $\mathcal{O}$ and
$\mathcal{U}$ in $GX$ there exists a sequence $t_{n}\rightarrow \infty$ such
that $t_{n}\cdot \mathcal{O}\cap \mathcal{U}\neq \emptyset$ for all $n.$
\end{definition}

\begin{definition}
\label{mix}The geodesic flow $\mathbb{R}\times GX\rightarrow GX$ is
topologically mixing if given any non-empty open sets $\mathcal{O}$ and
$\mathcal{U}$ in $GX$ there exists a real number $t_{0}>0$ such that for all
$\left \vert t\right \vert \geq t_{0},$ $t\cdot \mathcal{O}\cap \mathcal{U}%
\neq \emptyset.$
\end{definition}

The main theorem of this paper is the following

\begin{theorem}
\label{mainn}Let $X$ be the quotient of a CAT(0) space $\widetilde{X}$ by a
non-elementary discrete group of isometries $\Gamma$ acting freely on
$\widetilde{X}$ such that $\partial \widetilde{X}$ is connected and equal to
the limit set $\Lambda \left(  \Gamma \right)  .$ If conditions ($\Delta$), (U),
(C) and (D) stated above are satisfied, then the geodesic flow $\mathbb{R}%
\times GX\rightarrow GX$ is topologically mixing.
\end{theorem}

We will use the following results from \cite{[Coo]}. Let $Z$ be a proper 
$\delta$-hyperbolic geodesic metric space and $\Gamma$ a group of isometries of 
$Z$ acting properly discontinously on $Z$ such that the cardinality of the limit 
set $\Lambda (\Gamma)$ is infinite (in fact, the results below will be used in 
cases where $\Lambda (\Gamma) = \partial Z $).

\begin{proposition}
[{Cor. 4.2 and Cor. 6.3 Chapter II in \cite{[Coo]}}]\label{dense2} There
exists an orbit of $\Gamma$ dense in $\Lambda (\Gamma) \times \Lambda (\Gamma).$ 
In particular, for every $\xi \in \Lambda (\Gamma),$ 
the orbit $\Gamma \cdot \xi$ is dense in $\Lambda (\Gamma) .$
\end{proposition}

\begin{proposition}
[{Cor. 5.1, Chapter II in \cite{[Coo]}}]\label{dense3}The set
\[
\left \{  \left(  \phi(+\infty),\phi(-\infty)\right)  :\phi \in \Gamma
\mathrm{\ is\ a\ hyperbolic\ element}\right \}
\]
is dense in $\Lambda (\Gamma)\times \Lambda (\Gamma).$
\end{proposition}


Recall also that the boundary $\partial Y$ of a complete geodesic metric space
$Y$ can be defined, as a topological space, using Busemann functions as
explained in \cite[Ch. II, Sec. 1]{[Bal]}, where it is shown that the function
$\alpha:Y\times Y\times Y\rightarrow \mathbb{R}$ given by
\[
\alpha \left(  y,x,x^{\prime}\right)  :=d\left(  x^{\prime},y\right)  -d\left(
x,y\right)
\]
extends to a continuous function
\[
\alpha:\left(  Y\cup \partial Y\right)  \times Y\times Y\rightarrow \mathbb{R}%
\]
which is Lipschitz with respect to the second and third variable.

By \cite[Lemma 2.2 Ch.II]{[Bal]} and the discussion following it we have that
the topology given to $\partial \widetilde{X}$ via Busemann functions coincides
with the compact-open topology (given to $\partial \widetilde{X}$ using
geodesic rays and the fact that $\widetilde{X}$ is a CAT(0) space). Thus, we
obtain a continuous function
\[
\alpha:\left(  \partial \widetilde{X}\cup \widetilde{X}\right)  \times
\widetilde{X}\times \widetilde{X}\rightarrow \mathbb{R}%
\]
given by
\begin{equation}
\alpha \left(  y,x,x^{\prime}\right)  :=d\left(  x^{\prime},y\right)  -d\left(
x,y\right)  \label{generalized}%
\end{equation}
for $\left(  y,x,x^{\prime}\right)  \in \widetilde{X}\times \widetilde{X}%
\times \widetilde{X}$ and
\begin{equation}
\alpha \left(  \xi,x,x^{\prime}\right)  :=\lim_{n\rightarrow \infty}%
\alpha \left(  y_{n},x,x^{\prime}\right)  \label{generalizeddd}%
\end{equation}
for $\left(  \xi,x,x^{\prime}\right)  \in \partial \widetilde{X}\times
\widetilde{X}\times \widetilde{X}$ where $y_{n}\rightarrow \xi$ (see \cite[Ch.
II, Prop. 2.5]{[Bal]}).

This function, called the \textit{generalized Busemann function}, in fact,
generalizes the classical Busemann function whose definition makes sense in
our context: for arbitrary $\xi \in \partial \widetilde{X}$ and $x\in
\widetilde{X},$ the restriction
\[
\alpha \left(  \xi,x,\cdot \right)  \equiv \alpha|_{\left \{  \xi \right \}
\times \left \{  x\right \}  \times \widetilde{X}}%
\]
is simply the Busemann function associated to the unique geodesic ray from $x$
to $\xi.$

\noindent We will use the following facts about the generalized Busemann function.

\begin{lemma}
\label{busman}(a) The generalized Busemann function $\alpha$ is Lipschitz with
respect to the second and third variable with Lipschitz constant 1.
\newline(b) If $f,g\in G\widetilde{X}$ with $f\left(  -\infty \right)
=g\left(  +\infty \right)  $ then
\[
\alpha \left(  g\left(  +\infty \right)  ,g\left(  0\right)  ,f\left(  t\right)
\right)  =t+\alpha \left(  g\left(  +\infty \right)  ,g\left(  0\right)
,f\left(  0\right)  \right)  .
\]
(c) If $f,g\in G\widetilde{X}$ are asymptotic geodesics, then there exists a
unique re-parametrization $\overline{f}$ of $f$ such that $\alpha
\bigl(f\left(  +\infty \right)  ,\overline{f}\left(  0\right)  ,g\left(
0\right)  \bigr)=0.$
\end{lemma}

A\ proof of (a) can be found in \cite[Ch. II, Sec. 1]{[Bal]} and the proof
given for Lemma 2.3 in \cite{[CT3]} holds verbatim for (b) and (c).

\begin{definition}
\label{ssset}We say that a geodesic $h\in G\widetilde{X}$ belongs to the
stable set $W^{s}\left(  g\right)  $ of a geodesic $g$ if $g,h$ are
asymptotic. Two points $x,x^{\prime}\in \widetilde{X}$ are said to be
equidistant from a point $\xi \in \partial \widetilde{X}$ if $\alpha \left(
\xi,x,x^{\prime}\right)  =0.$

We say that a geodesic $h\in G\widetilde{X}$ belongs to the strong stable set
$W^{ss}\left(  g\right)  $ of a geodesic $g$ if $h\in W^{s}\left(  g\right)  $
and $g\left(  0\right)  ,h\left(  0\right)  $ are equidistant from $g\left(
\infty \right)  =h\left(  \infty \right)  $.

Similarly, if $h,g\in GX,$ we say that $h\in W^{ss}\left(  g\right)  $
$\bigl($respectively $W^{s}\left(  g\right)  \bigr)$ if there exist lifts
$\widetilde{h},\widetilde{g}\in G\widetilde{X}$ of $h,g$ such that
$\widetilde{h}\in W^{ss}\left(  \widetilde{g}\right)  $ $\bigl($respectively
$W^{s}\left(  \widetilde{g}\right)  \bigr).$
\end{definition}

We next restate Condition (C) using the terminology of strong stable sets.

\begin{proposition}
\label{asy}Let $f,g\in G\widetilde{X}$ with $f\in W^{ss}\left(  g\right)  .$
Assume $f\left(  +\infty \right)  =g\left(  +\infty \right)  \in \partial
\widetilde{X} \setminus F_{h}^{nu},$ that is, if $h\in W^{s}\left(  g\right)
$ then $h$ is not a non-unique closed geodesic. Then $\mathrm{lim}%
_{t\rightarrow \infty}d\bigl(f\left(  t\right)  ,g\left(  t\right)  \bigr)=0.$
\end{proposition}

The proof of the above proposition is identical with the one given in
\cite[Prop. 2.2]{[CT3]}. We conclude this section with the following

\begin{lemma}
\label{eberlein210}Let $x,y\in \widetilde{X}$ and $\xi \in \partial \widetilde{X}$
with $\alpha \left(  \xi,x,y\right)  =0.$ For any open set $O$ in
$\widetilde{X}$ containing $y,$ there exist open sets $C$ and $D$ of
$\widetilde{X}$ and $\partial \widetilde{X}$ respectively such that $\left(
x,\xi \right)  \in C\times D$ and for every $\left(  x^{\prime},\xi^{\prime
}\right)  \in C\times D$ there exists $y^{\prime}\in O$ with $\alpha \left(
\xi^{\prime},x^{\prime},y^{\prime}\right)  =0.$
\end{lemma}

\begin{proof}
Given an open set $O$ containing $y,$ choose $\varepsilon>0$ so that the open
ball $B\left(  y,\varepsilon \right)  \subset O.$ As $\alpha \left(
\xi,x,y\right)  =0,$ by continuity of $\alpha$ we may find open sets
$C\subset \widetilde{X}$ and $D\subset \partial \widetilde{X}$ such that $\left(
x,\xi \right)  \in C\times D$ and
\[
\left(  x^{\prime},\xi^{\prime}\right)  \in C\times D\Rightarrow \left \vert
\alpha \left(  \xi^{\prime},x^{\prime},y\right)  \right \vert <\varepsilon.
\]
These are the desired open sets.\newline Given $\left(  x^{\prime},\xi
^{\prime}\right)  \in C\times D,$ let $r^{\prime}$ be the geodesic ray with
$r^{\prime}(0)=y$ and $r^{\prime}(+\infty)=\xi^{\prime}.$ Denote by
$f^{\prime}$ any geodesic line which extends $r^{\prime}.$ By Lemma
\ref{busman}(b)
\[
\bigm \vert \alpha \left(  \xi^{\prime},x^{\prime},f^{\prime}(t)\right)
-\alpha \left(  \xi^{\prime},x^{\prime},f^{\prime}(0)\right)  \bigm \vert=|t|.
\]
Let $t_{0}=\alpha \left(  \xi^{\prime},x^{\prime},y\right)  .$ Then $\lvert
t_{0}\rvert<\varepsilon$ and $\alpha \left(  \xi^{\prime},x^{\prime},f^{\prime
}(t_{0})\right)  =0.$ Since $f^{\prime}(t)\in O$ for $\lvert t\rvert
<\varepsilon$ we have $y^{\prime}:=f^{\prime}(t_{0})\in O$ and\newline%
\hspace*{5cm}$\alpha \left(  \xi^{\prime},x^{\prime},y^{\prime}\right)
=\alpha \left(  \xi^{\prime},x^{\prime},f^{\prime}(t_{0})\right)  =0.$
\end{proof}

\subsection{Properties of geodesics and geodesic rays\label{sectionproperties}%
}

Let $GX$ be the space of all local isometric maps $g:\mathbb{R}\rightarrow X.$
As usual, the image of such a $g$ will be referred to as a geodesic in $X.$
Consider also the space $G\widetilde{X}$ of all isometric maps $g:\mathbb{R}%
\rightarrow \widetilde{X}.$ Both spaces $GX$ and $G\widetilde{X}$ are equipped
with the compact-open topology. Moreover the space $G\widetilde{X}$ with the
compact-open topology is metrizable (see \cite[8.3.B]{[Gro]}) and second countable.

We will denote by $p$ both projections $\widetilde{X}\rightarrow X$ and
$G\widetilde{X}\rightarrow GX.$ Denote by $R\widetilde{X}$ the set of all
geodesic rays in $\widetilde{X},$ that is, the set of all isometric maps
$r:\left[  0,\infty \right)  \rightarrow \widetilde{X}$ equipped with the
compact open topology.

\begin{proposition}
\label{rvfunction}The function $\varrho:R\widetilde{X}\rightarrow \widetilde
{X}\times \partial \widetilde{X}$ given by
\[
\varrho \left(  r\right)  =\left(  r\left(  0\right)  ,r\left(  \infty \right)
\right)  ,
\]
where $r\left(  \infty \right)  $ denotes the unique boundary point determined
by $r,$ is a homeomorphism.
\end{proposition}

\begin{proof}
By uniqueness of geodesic rays the inverse function $\varrho^{-1}$ is well
defined for all $\left(  x,\xi \right)  \in \widetilde{X}\times \partial
\widetilde{X}.$

We first show continuity of $\varrho^{-1}.$ Let $\left(  x_{0},\xi_{0}\right)
\in \widetilde{X}\times \partial \widetilde{X}$ and $\left \{  x_{n}\right \}
\subset \widetilde{X},$ $\left \{  \xi_{n}\right \}  \subset \partial \widetilde
{X}$ be sequences with $x_{n}\rightarrow x_{0}$ and $\xi_{n}\rightarrow \xi
_{0}.$ Denote by $r_{n},$ $n\geq0$ the unique geodesic ray with $r_{n}\left(
0\right)  =x_{n}$ and $r_{n}\left(  \infty \right)  =\xi_{n}.$ Similarly,
denote by $q_{n},$ $n\geq1$ the unique geodesic ray with $q_{n}\left(
0\right)  =x_{0}$ and $q_{n}\left(  \infty \right)  =\xi_{n}.$ The assumption
$\xi_{n}\rightarrow \xi_{0}$ means, by definition, that
\begin{equation}
q_{n}\rightarrow r_{0} \label{rsyng}%
\end{equation}
and we need to show $r_{n}\rightarrow r_{0}.$ For each $n\in \mathbb{N},$ the
geodesic rays $q_{n}$ and $r_{n}$ are asymptotic, hence, the distance function
$t\rightarrow d\left(  q_{n}\left(  t\right)  ,r_{n}\left(  t\right)  \right)
,$ $t\geq0$ is convex (see \cite[Ch.I, Proposition 5.4]{[Bal]}) and bounded.
Therefore, it is decreasing and $d\left(  x_{0},x_{n}\right)  $ is an upper
bound for all $t\geq0$ because
\begin{equation}
d\left(  x_{0},x_{n}\right)  =d\left(  q_{n}\left(  0\right)  ,r_{n}\left(
0\right)  \right)  \geq d\left(  q_{n}\left(  t\right)  ,r_{n}\left(
t\right)  \right)  \label{rsdecreasing}%
\end{equation}
Let $\mathcal{O}$ be a neighborhood of $r_{0}$ of the form
\begin{equation}
\mathcal{O}\left(  r_{0},K,\varepsilon \right)  =\left \{  r^{\prime}\in
R\widetilde{X}\bigm \vert d\left(  r^{\prime}\left(  t\right)  ,r_{0}\left(
t\right)  \right)  <\varepsilon \mathrm{\ for\ all\ }t\in \left[  0,K\right]
\right \}  . \label{rnstrict}%
\end{equation}
Find $n_{1}\in \mathbb{N}$ such that $d\left(  x_{0},x_{n}\right)
<\varepsilon/2$ for all $n>n_{1}$ which, by (\ref{rsdecreasing}), yields
\[
d\left(  q_{n}\left(  t\right)  ,r_{n}\left(  t\right)  \right)
<\frac{\varepsilon}{2}\mathrm{\ for\ all\ }n>n_{1}\mathrm{\ and\ }t\in \left[
0,K\right]  .
\]
As $q_{n}\rightarrow r_{0}$ we may find $n_{2}\in \mathbb{N}$ such that
$q_{n}\in \mathcal{O}\left(  r_{0},K,\varepsilon/2\right)  $ for all $n>n_{2}$
which means%
\[
d\left(  q_{n}\left(  t\right)  ,r_{0}\left(  t\right)  \right)
<\frac{\varepsilon}{2}\mathrm{\ for\ all\ }n>n_{2}\mathrm{\ and\ }t\in \left[
0,K\right]  .
\]
Combining the last two inequalities we have
\[
d\left(  r_{n}\left(  t\right)  ,r_{0}\left(  t\right)  \right)
<\frac{\varepsilon}{2}+\frac{\varepsilon}{3}\mathrm{\ for\ all\ }%
n>\max \left \{  n_{1},n_{2}\right \}  \mathrm{\ and\ }t\in \left[  0,K\right]
\]
which shows that $r_{n}\in \mathcal{O}$ for $n$ large enough. Thus,
$r_{n}\rightarrow r_{0}$ as desired.

For the continuity of $\varrho,$ let $\left \{  r_{n}\right \}  \subset
R\widetilde{X}$ be a sequence converging to a geodesic ray $r_{0}.$ Clearly,
$r_{n}\left(  0\right)  \rightarrow r_{0}\left(  0\right)  $ and we need to
check that $r_{n}\left(  +\infty \right)  \rightarrow r_{0}\left(
+\infty \right)  .$ This amounts to verifying that
\[
q_{n}\rightarrow r_{0},
\]
where $q_{n},$ $n\geq1$ is the unique geodesic ray with $q_{n}\left(
0\right)  =x_{0}$ and $q_{n}\left(  \infty \right)  =r_{n}\left(
+\infty \right)  .$ Since for each $n,$ the geodesic rays $r_{n}$ and $q_{n}$
are asymptotic, an argument similar to the one given above for $\varrho^{-1},$
shows that for an arbitrary neighborhood $\mathcal{O}$ of $r_{0},$ $q_{n}%
\in \mathcal{O}$ for $n$ large enough.
\end{proof}

\begin{proposition}
\label{bakis}Let $f$ be a unique geodesic in $G\widetilde{X}$ and $\left \{
f_{n}\right \}  \subset G\widetilde{X}$ a sequence of geodesics with
$f_{n}\left(  +\infty \right)  \rightarrow f\left(  +\infty \right)  $ and
$f_{n}\left(  -\infty \right)  \rightarrow f\left(  -\infty \right)  .$ Then we
may re-parametrize $\left \{  f_{n}\right \}  $ such that $f_{n}\rightarrow f.$
\end{proposition}

\begin{proof}
Fix $x_{0}:=f\left(  0\right)  $ as base point and choose sequences $\left \{
x_{n}\right \}  ,\left \{  y_{n}\right \}  $ with $x_{n},y_{n}\in
\operatorname{Im}f_{n}$ for each $n,$ such that
\[
x_{n}\rightarrow f\left(  +\infty \right)  ,y_{n}\rightarrow f\left(
-\infty \right)  \, \, \mathrm{and}\, \,d\left(  x_{0},x_{n}\right)  =d\left(
x_{0},y_{n}\right)
\]
for all $n.$ Denote by $m_{n}$ the midpoint of the segment $\left[
x_{n},y_{n}\right]  \subset \operatorname{Im}f_{n}$ and, by passing if
necessary to a subsequence, we may assume that $\left \{  m_{n}\right \}  $
converges to $m\in \widetilde{X}\cup \partial \widetilde{X}.$ \newline Recall
that equivalence classes of unbounded sequences can be used, via the Gromov
product
\[
\left(  x,y\right)  _{x_{0}}:=\frac{1}{2}\left(  d\left(  x_{0},x\right)
+d\left(  x_{0},y\right)  -d\left(  x,y\right)  \right)
\]
to define the boundary of a hyperbolic space (see \cite{[CDP]} Ch.2, \S 2). We
examine the following three cases:

\begin{itemize}
\item $\left \{  m_{n}\right \}  $ is bounded, that is, $m_{n}\rightarrow
m\in \widetilde{X},$

\item $\left \{  m_{n} \right \}  $ is unbounded and equivalent to $\left \{
x_{n}\right \}  $ (and, hence, to $\left \{  y_{n}\right \}  $), that is,
$\left(  m_{n},x_{n}\right)  _{x_{0}}\rightarrow \infty$ as $n\rightarrow
\infty,$

\item $\left \{  m_{n} \right \}  $ is unbounded and $\left \{  m_{n}\right \}  $
is not equivalent to $\left \{  x_{n}\right \}  $ (and, hence, neither to
$\left \{  y_{n}\right \}  $).
\end{itemize}

In the first case, since $f$ is unique, $m$ must belong to $\operatorname{Im}%
f,$ hence, using the real numbers $t_{m}$ where $f\left(  t_{m}\right)  =m$
and $f_{n}\left(  t_{m_{n}}\right)  =m_{n}$ with the appropriate signs
according to orientation, we obtain the desired parametrization of each
$f_{n}$. In the second case, by the choice of $\left \{  x_{n}\right \}
,\left \{  y_{n}\right \}  ,$ we have $\left(  m_{n},x_{n}\right)  _{x_{0}%
}=\left(  m_{n},y_{n}\right)  _{x_{0}}.$ Thus, all three sequences $\left \{
x_{n}\right \}  ,\left \{  y_{n}\right \}  $ and $\left \{  m_{n}\right \}  $ are
equivalent, that is, they define the same boundary point, a contradiction
since $f\left(  +\infty \right)  \neq f\left(  -\infty \right)  .$ In the third
case $\left \{  m_{n}\right \}  $ converges to a point $m\in \partial
\widetilde{X}$ with $m\neq f\left(  \pm \infty \right)  .$ This case cannot
occur either. Indeed, as $\partial \widetilde{X}$ is metrizable, we may find
neighborhoods $O$ and $U$ of $f\left(  +\infty \right)  $ and $f\left(
-\infty \right)  $ respectively, such that $m\notin O\cup U.$ Set
\[
Q\left(  O,U\right)  :=\left \{  x\in \widetilde{X}\bigm \vert \, \,x\in
\operatorname{Im}g,\, \mathrm{for}\, \, \mathrm{some}\, \,g\in G\widetilde
{X}\, \, \mathrm{with}\, \,g\left(  \pm \infty \right)  \in O\cup U\right \}  .
\]
Then, by \cite[Part B, Lemma 3]{[Coo]}, $O\cup U$ is the accumulation set of
$Q\left(  O,U\right)  $ in $\partial \widetilde{X},$ a contradiction, since
$m\notin O\cup U.$
\end{proof}

\section{Mixing of the geodesic flow}

\subsection{Topological transitivity}

It is apparent that topological mixing implies topological transitivity.
However, in the proof of topological mixing below we will need a property
equivalent to topological transitivity, namely, that $\overline{W^{s}\left(
f\right)  }=GX$ for any $f\in GX.$ We are omitting the proof of the
equivalence since it will not be used in the sequel. In this section we will
establish this property (see Proposition \ref{all_stable} below). We need the
following\\[3mm]

\noindent \textbf{Duality Condition}: For each $f\in G\widetilde{X},$ there
exists a sequence of isometries $\left \{  \phi_{n}\right \}  _{n\in \mathbb{N}%
}\subset \Gamma \equiv \pi_{1}\left(  X\right)  $ such that $\phi_{n}\left(
x\right)  \rightarrow f\left(  +\infty \right)  $ and $\phi_{n}^{-1}\left(
x\right)  \rightarrow f\left(  -\infty \right)  $ for some (hence any)
$x\in \widetilde{X}.$

Topological transitivity for our class of spaces will then follow from the
following Theorem found in \cite[Th. 2.3 Ch. III]{[Bal]}.

\begin{theorem}
\label{ballduoteliatria}Let $Y$ be a geodesically complete separable Hadamard
space and $\Gamma$ a group of isometries of $Y$ satisfying the duality
condition. Then the following are equivalent:\newline(a) the geodesic flow is
topologically transitive mod$\, \, \Gamma.$ \newline(b) for some $\xi
\in \partial Y,$ the orbit $\Gamma \cdot \xi$ is dense in $\partial Y.$
\end{theorem}

\begin{lemma}
\label{dualityLemma}$\widetilde{X}$ satisfies the duality condition.
\end{lemma}

\begin{proof}
If $f\in G\widetilde{X}$ is closed we may consider $\left \{  \phi_{n}\right \}
$ to be powers of the hyperbolic isometry corresponding to $f.$ Then clearly
$\phi_{n}\left(  f\left(  0\right)  \right)  \rightarrow f\left(
+\infty \right)  $ and $\phi_{n}^{-1}\left(  f\left(  0\right)  \right)
\rightarrow f\left(  -\infty \right)  .$

Suppose $f\in G\widetilde{X}$ is non closed. By Proposition \ref{dense3},
there exists a sequence of closed geodesics $c_{n}$ such that $c_{n}\left(
+\infty \right)  \rightarrow f\left(  +\infty \right)  $ and $c_{n}\left(
-\infty \right)  \rightarrow f\left(  -\infty \right)  .$ Using Proposition
\ref{bakis} and changing appropriately the parametrizations of each $c_{n},$
we obtain $c_{n}\rightarrow f.$ We may alter the period $t_{n}$ of each
$c_{n}$ so that $t_{n}\nearrow+\infty$ as $n\rightarrow \infty.$ Set $\phi_{n}$
to be the isometry which corresponds to translating $c_{n}$ by $t_{n}.$ Then,
$\phi_{n}\left(  c_{n}\left(  0\right)  \right)  \rightarrow f\left(
+\infty \right)  $ and, since $f\left(  0\right)  $ is at bounded distance from
$c_{n}\left(  0\right)  $ for all $n,$ it follows that $\phi_{n}\left(
f\left(  0\right)  \right)  \rightarrow f\left(  +\infty \right)  .$ Similarly
we show that $\phi_{n}^{-1}\left(  f\left(  0\right)  \right)  \rightarrow
f\left(  -\infty \right)  .$
\end{proof}

\begin{theorem}
\label{orbit}There exists a geodesic $\gamma$ in $GX$ whose orbit
$\mathbb{R}\gamma$ under the geodesic flow is dense in $GX.$ Equivalently, the
geodesic flow is topologically transitive.
\end{theorem}

\begin{proof}
\noindent Equivalence of the two statements is a general fact which follows
from separability of $\widetilde{X}$ and 2nd countability of the topology of
$G\widetilde{X}$ (see \cite[Remark 2.2 Ch. III]{[Bal]}). By Proposition
\ref{dense2}, for any $\xi \in \partial \widetilde{X}$ the orbit $\Gamma \cdot \xi$
is dense in $\partial \widetilde{X}.$ Moreover, by the above lemma,
$\widetilde{X}$ satisfies the duality condition, thus, by the above mentioned
Theorem \ref{ballduoteliatria} from \cite{[Bal]}, the geodesic flow is
topologically transitive.
\end{proof}

Observe that, in particular, the image of such a geodesic $\gamma$ is a dense
subset of $X.$ Therefore, the geodesic $\gamma$ whose orbit is dense in $GX$
cannot be a closed geodesic. We will need the following

\begin{corollary}
\label{denseorbitunc}There exists a geodesic $\gamma$ in $GX$ whose orbit
$\mathbb{R}\gamma$ under the geodesic flow is dense in $GX$ and, in addition,
$\widetilde{\gamma}\left(  +\infty \right)  \notin F_{h}^{nu}$ for some, hence
any, lift $\widetilde{\gamma}$ of $\gamma.$
\end{corollary}

\begin{proof}
We first show that the cardinality of the set $D=\left \{  \gamma
\bigm \vert \overline{\mathbb{R}\gamma}=GX\right \}  $ is uncountably infinite.
To check this, observe that $D$ is just the intersection
\[
D=\cap_{\mathcal{B}\in \underline{\mathfrak{B}}}\mathbb{R}\mathcal{B}%
\]
where $\underline{\mathfrak{B}}$ is a countable basis for $GX$ with
$\emptyset$ excluded. Since the geodesic flow is topologically transitive,
each $\mathbb{R}\mathcal{B}$ is dense and, clearly, open and non-empty. By
Baire's theorem, $D$ is non-empty and if $D=\left \{  \gamma_{1},\gamma
_{2},\ldots \right \}  $ were countable, then the countable intersection
\[
\cap_{\mathcal{B}\in \underline{\mathfrak{B}}^{\prime}}\mathbb{R}\mathcal{B}%
\]
where $\underline{\mathfrak{B}}^{\prime}=\underline{\mathfrak{B}}\cup \left(
\cup_{i=1}^{\infty}GX\setminus \left \{  \gamma_{i}\right \}  \right)  $ would be
empty, contradicting Baire's theorem. The Corollary now follows from the fact
that $F_{h}^{nu}\times F_{h}^{nu}$ is countable, thus, there exist $\gamma \in
D$ such that $\left(  \widetilde{\gamma}\left(  -\infty \right)  ,\widetilde
{\gamma}\left(  +\infty \right)  \right)  \notin F_{h}^{nu}\times F_{h}^{nu}.$
In other words, there exists a geodesic $\gamma$ with dense orbit in $GX$
whose lift to $G\widetilde{X}$ has at least one of its limit points in
$\partial \widetilde{X}\setminus F_{h}^{nu}.$ By replacing $\gamma$ with
$-\gamma$ we may assume that $\widetilde{\gamma}\left(  +\infty \right)  \notin
F_{h}^{nu}.$
\end{proof}

\begin{proposition}
\label{all_stable}For any $f\in GX,$ $\overline{W^{s}\left(  f\right)  }=GX.$
\end{proposition}

\begin{proof}
Let $g\in GX$ be arbitrary and pick lifts $\widetilde{g}\in G\widetilde{X}$ of
$g$ and $\widetilde{f}\in G\widetilde{X}$ of $f.$ Theorem \ref{orbit} provides
a geodesic $\widetilde{\gamma}\in G\widetilde{X}$ and sequences $\{t_{n}\}$ in
$\mathbb{R}$ and $\left \{  \phi_{n}\right \}  $ in $\Gamma$ such that $\phi
_{n}\left(  t_{n}\cdot \widetilde{\gamma}\right)  \rightarrow \widetilde{g}.$
Set $\widetilde{\gamma}_{n}:=\phi_{n}\left(  t_{n}\cdot \widetilde{\gamma
}\right)  .$

Since the orbit $\Gamma \cdot \widetilde{f}\left(  +\infty \right)  $ is dense in
$\partial \widetilde{X}$ we may pick, for each fixed $n$, a sequence $\left \{
\phi_{n,k}\right \}  _{k=1}^{\infty}\subset \Gamma$ such that $\phi_{n,k}\left(
\widetilde{f}\left(  +\infty \right)  \right)  \rightarrow \widetilde{\gamma
}_{n}\left(  +\infty \right)  .$ For all $k,$ consider geodesics $\widetilde
{g}_{n,k}$ with $\widetilde{g}_{n,k}\left(  +\infty \right)  =\phi_{n,k}\left(
\widetilde{f}\left(  +\infty \right)  \right)  $ and $\widetilde{g}%
_{n,k}\left(  -\infty \right)  =\widetilde{\gamma}_{n}\left(  -\infty \right)
.$ Clearly $g_{n,k}=p\left(  \widetilde{g}_{n,k}\right)  \in W^{s}\left(
f\right)  $ for all $k,n\in \mathbb{N}$ and by a diagonal argument we obtain a
sequence $g_{n,k(n)}=p\left(  \widetilde{g}_{n,k(n)}\right)  $ which, up to
appropriate parametrization, converges to $g=p\left(  \widetilde{g}\right)  .$
\end{proof}

\subsection{Proof of topological mixing\label{sectionmixing}}

For the proof of topological mixing for the geodesic flow on $X,$ we will
closely follow the notation and the analogous proof for CAT(-1) spaces in
\cite{[CT3]} which, in fact, follows the steps of Eberlein's work (cf
\cite{[Ebe2]}). For the proofs of the following Lemmata, we will refer to the
corresponding proofs in \cite{[CT3]} and deal only with the issues arising
from the non-unique closed geodesics.

\begin{lemma}
\label{contofsets} (a) For any $g\in GX$ and $t\in \mathbb{R},$ $\overline
{W^{ss}\left(  t\cdot g\right)  }=t\cdot \left(  \overline{W^{ss}\left(
g\right)  }\right)  .$

(b) Let $g,h\in GX$ with $h\in W^{ss}\left(  g\right)  $ and $\mathcal{O}%
\subset GX$ an open neighborhood of $h.$ Then there exists an open
neighborhood $\mathcal{A}$ containing $g$ such that for any $g_{1}%
\in \mathcal{A},$ $W^{ss}\left(  g_{1}\right)  \cap \mathcal{O}\neq \emptyset.$

(c) If $h\in \overline{W^{ss}\left(  g\right)  },$ then $\overline
{W^{ss}\left(  h\right)  }\subset \overline{W^{ss}\left(  g\right)  }.$
\end{lemma}

\begin{proof}
(a) If $h\in \overline{W^{ss}\left(  t\cdot g\right)  }$, there exist a
sequence $\left \{  h_{n}\right \}  _{n\in \mathbb{N}}\subset W^{ss}\left(
t\cdot g\right)  $ with $h_{n}\rightarrow h.$ It is clear from the definitions
that $\left(  -t\right)  \cdot h_{n}\rightarrow \left(  -t\right)  \cdot h$ and
$\left \{  \left(  -t\right)  \cdot h_{n}\right \}  _{n\in \mathbb{N}}\subset
W^{ss}\left(  g\right)  .$ This shows that $\left(  -t\right)  \cdot
h\in \overline{W^{ss}\left(  g\right)  }$ and, hence, $h=t\cdot \bigl(\left(
-t\right)  \cdot h\bigr)\in t\cdot \left(  \overline{W^{ss}\left(  g\right)
}\right)  .$ For the converse inclusion, let $h\in t\cdot \left(
\overline{W^{ss}\left(  g\right)  }\right)  .$ This means that there exist a
sequence $\left \{  h_{n}\right \}  _{n\in \mathbb{N}}\subset W^{ss}\left(
g\right)  $ with $t\cdot h_{n}\rightarrow h.$ Clearly, $t\cdot h_{n}\in
W^{ss}\left(  t\cdot g\right)  $ hence $h\in \overline{W^{ss}\left(  t\cdot
g\right)  }.$

(b) Lift $g$ and $h$ to geodesics $\widetilde{g}$ and $\widetilde{h}$ in
$G\widetilde{X}$ such that $\widetilde{h}\in W^{ss}\left(  \widetilde
{g}\right)  $ and consider an open neighborhood $\widetilde{\mathcal{O}%
}\subset G\widetilde{X}$ of $\widetilde{h}$ such that $p\left(  \widetilde
{\mathcal{O}}\right)  \subset \mathcal{O}.$ We will show that there exists an
open neighborhood $\widetilde{\mathcal{A}}$ containing $\widetilde{g}$ such
that for any $\widetilde{g}_{1}\in \widetilde{\mathcal{A}},$ $W^{ss}\left(
\widetilde{g}_{1}\right)  \cap \widetilde{\mathcal{O}}\neq \emptyset.$ Then
$\mathcal{A}=p\left(  \widetilde{\mathcal{A}}\right)  $ would be the desired
neighborhood of $g=p\left(  \widetilde{g}\right)  .$

We may assume that $\widetilde{\mathcal{O}}$ is of the form
\[
\widetilde{\mathcal{O}}\left(  \widetilde{h},K,\varepsilon \right)  =\left \{
\widetilde{f}\in G\widetilde{X}\bigm \vert d\left(  \widetilde{f}\left(
t\right)  ,\widetilde{h}\left(  t\right)  \right)  <\varepsilon
\mathrm{\ for\ all\ }t\in \left[  -K,K\right]  \right \}  .
\]
Consider the open neighborhood
\[
\widetilde{\mathcal{O}}_{R}=\left \{  r\in R\widetilde{X}\bigm \vert
r=\widetilde{f}|_{\left[  -K,\infty \right)  }\mathrm{\ for\ some\ }%
\widetilde{f}\in \widetilde{\mathcal{O}}\right \}
\]
of $R\widetilde{X}.$ Clearly, $\alpha \left(  \xi,\widetilde{h}\left(
0\right)  ,\widetilde{g}\left(  0\right)  \right)  =0$ where $\xi
=\widetilde{g}\left(  +\infty \right)  =\widetilde{h}\left(  +\infty \right)  .$
By Proposition \ref{rvfunction} we may choose open sets $A$ and $B$ of
$\widetilde{X}$ and $\partial \widetilde{X}$ respectively, such that $\left(
\widetilde{h}\left(  -K\right)  ,\xi \right)  \in A\times B$ and $\varrho
^{-1}\left(  A\times B\right)  \subset \widetilde{\mathcal{O}}_{R}$ where
$\varrho$ is the function produced in Proposition \ref{rvfunction}. By Lemma
\ref{eberlein210}, we may choose open sets $C$ and $D$ of $\widetilde{X}$ and
$\partial \widetilde{X}$ respectively, such that $\left(  \widetilde{g}\left(
-K\right)  ,\xi \right)  \in C\times D$ and for every geodesic ray
$r_{g}:\left[  -K,\infty \right)  \rightarrow \widetilde{X}$ with $r_{g}\left(
-K\right)  \in C$ and $r_{g}\left(  +\infty \right)  \in D,$ there exists a
geodesic ray $r_{h}:\left[  -K,\infty \right)  \rightarrow \widetilde{X}$ with
$r_{h}\left(  -K\right)  \in A,$ $r_{g}\left(  +\infty \right)  =r_{h}\left(
+\infty \right)  $ and $\alpha \left(  r_{g}\left(  +\infty \right)
,r_{h}\left(  -K\right)  ,r_{g}\left(  -K\right)  \right)  =0$, in other
words, $r_{h}\in W^{ss}\left(  r_{g}\right)  .$

The inverse image $\varrho^{-1}\left(  C\times \left(  B\cap D\right)  \right)
$ is an open neighborhood in $R\widetilde{X}$ containing $\widetilde
{g}|_{\left[  -K,\infty \right)  }.$ Extend all geodesic rays in $\varrho
^{-1}\left(  C\times \left(  B\cap D\right)  \right)  $ to geodesic lines in
order to obtain $\widetilde{\mathcal{A}}\subset G\widetilde{X}$ containing
$\widetilde{g}.$

Every geodesic $\widetilde{g}_{1}\in \widetilde{\mathcal{A}}$ determines a
geodesic ray $r_{g_{1}}\in \varrho^{-1}\left(  C\times \left(  B\cap D\right)
\right)  $ for which we have shown that there exists a geodesic ray $r_{h_{1}%
}\in \varrho^{-1}\left(  A\times B\right)  .$ Extending $r_{h_{1}}$ to a
geodesic line we obtain a geodesic $\widetilde{h}_{1}\in \widetilde
{\mathcal{O}}$ with $\widetilde{h}_{1}\in W^{ss}\left(  \widetilde{g}%
_{1}\right)  ,$ thus $W^{ss}\left(  \widetilde{g}_{1}\right)  \cap
\widetilde{\mathcal{O}}\neq \emptyset.$

(c) Let $g^{\ast}$ be an arbitrary element in $\overline{W^{ss}\left(
h\right)  }$ and $\mathcal{O}$ an arbitrary open neighborhood of $g^{\ast}.$
We will show that $g^{\ast}\in \overline{W^{ss}\left(  g\right)  }.$ Since
$W^{ss}\left(  h\right)  \cap \mathcal{O}\neq \emptyset$ we may choose, by part
(b), an open neighborhood $\mathcal{A}$ of $h$ such that, for every
$f\in \mathcal{A},$ $W^{ss}\left(  f\right)  \cap \mathcal{O}\neq \emptyset.$
Since $h\in \overline{W^{ss}\left(  g\right)  },$ there exists $g_{1}\in
W^{ss}\left(  g\right)  \cap \mathcal{A}$ and, thus, $W^{ss}\left(
g_{1}\right)  \cap \mathcal{O}=W^{ss}\left(  g\right)  \cap \mathcal{O}%
\neq \emptyset.$ Since $\mathcal{O}$ was arbitrary, $g^{\ast}\in \overline
{W^{ss}\left(  g\right)  }$ as required.
\end{proof}

\begin{lemma}
\label{karafinalL}Let $f\in G\widetilde{X}$ be unique, that is, either $p(f)$
is non-closed or, $p\left(  f\right)  \in GX$ is closed and unique in its
homotopy class, and $\mathcal{O}$ a neighborhood of $f.$ \newline(a) There
exists a neighborhood $O$ in $\partial \widetilde{X}$ of $f\left(
+\infty \right)  $ such that for any geodesic $g$ with $g\left(  +\infty
\right)  \in O$ and $g\left(  -\infty \right)  =f\left(  -\infty \right)  $
there exists a re-parametrization $\overline{g}$ of $g$ with $\overline{g}%
\in \mathcal{O}.$\newline(b) If $\left \{  \xi_{n}\right \}  $ is a sequence in
$\partial \widetilde{X}$ with $\xi_{n}\rightarrow f\left(  +\infty \right)  $
and $\left \{  f_{n}\right \}  $ a sequence of geodesics with $f_{n}\left(
+\infty \right)  =\xi_{n}$ and $f_{n}\left(  -\infty \right)  =f\left(
-\infty \right)  ,$ we may re-parametrize $\left \{  f_{n}\right \}  $ such that
$f_{n}\rightarrow f.$
\end{lemma}

\begin{proof}
(a) The proof of part (a) follows from (b). To see this, assume the result
does not hold. Then for a decreasing sequence of open neighborhoods
$O_{n}\searrow f\left(  +\infty \right)  $ there must exist $\xi_{n}\in O_{n}$
such that any geodesic $f_{n}$ with $f_{n}\left(  +\infty \right)  =\xi_{n}$
and $f_{n}\left(  -\infty \right)  =f\left(  -\infty \right)  $ has the property
$f_{n}\notin \mathcal{O}.$ In particular $\left \{  f_{n}\right \}  $ does not
converge to $f$ contradicting (b).

(b) This is a special case of Proposition \ref{bakis}.
\end{proof}

\begin{proposition}
\label{ssclosure_somegeodesic}There exists a geodesic $g\in GX$ such that
$\overline{W^{ss}\left(  g\right)  }=GX.$
\end{proposition}

\begin{proof}
We will follow the line of proof of Proposition 4.1 in \cite{[CT3]}. In that
setup geodesic lines are uniquely determined by their boundary points, so
Conditions (u) and (c) stated in the beginning of the Introduction hold. The
modification will consist of the following: every geodesic which comes into
play will be replaced by a (unique) geodesic whose limit point belongs to
$\partial \widetilde{X} \setminus F_{h}^{nu}$ so that conditions (U) and (C)
can be applied.

To prove the Proposition, it suffices to show that
\begin{equation}%
\begin{aligned}
\mathrm{for}\, \, \mathrm{any}\, \, \mathrm{open}\, \, \mathcal{O}\, \,
\mathrm{and}\, \, \, \mathcal{U}  &  \subseteq GX\,,\, \\
\mathrm{there}\, \, \mathrm{exists}\, \, \,g  &  \in \mathcal{O}%
\mathrm{\ such\ that\ }W^{ss}\left(  g\right)  \cap \mathcal{U}\neq \emptyset.
\end{aligned}
\label{basicoldstep}%
\end{equation}
Then, using a countable basis $\left \{  \mathcal{O}_{n}\right \}
_{n\in \mathbb{N}}$ for the topology of $GX$ the proof is completed by a
standard topological argument (cf. \cite[Theorem 5.2]{[Ebe2]}).

Let $\mathcal{O},\mathcal{U}\subseteq GX\,$ be arbitrary open sets. Pick $f\in
p^{-1}\left(  \mathcal{O}\right)  $ such that $f$ is non-closed. Similarly,
choose $h\in p^{-1}\left(  \mathcal{U}\right)  $ such that $h$ is not closed.

By condition (U) $f$ (resp. $h$) is unique, thus, by Lemma \ref{karafinalL}%
(a), there exists connected open neighborhood $O_{f}\subset \partial
\widetilde{X}$ of $f\left(  +\infty \right)  $ (resp. $U_{h}$ of $h\left(
+\infty \right)  )$ such that for every $\xi \in O_{f}$ (resp. $\xi \in U_{h}$)
there exists a geodesic with boundary points $\xi,f\left(  -\infty \right)  $
(resp. $\xi,h\left(  -\infty \right)  $) which belongs to $p^{-1}\left(
\mathcal{O}\right)  $ (resp. $p^{-1}\left(  \mathcal{U}\right)  ).$

By condition (D), there exists a closed and unique geodesic $\beta$ such that
\begin{equation}
\left \{  \beta \left(  +\infty \right)  ,\beta \left(  -\infty \right)  \right \}
\subset F_{h}^{u} \label{superkarafinal}%
\end{equation}
and
\[
\left(  \beta \left(  +\infty \right)  ,\beta \left(  -\infty \right)  \right)
\in O_{f}\times U_{h}.
\]
By the choice of $O_{f}$ (Lemma \ref{karafinalL}(a)) there exists a geodesic
joining $\beta \left(  +\infty \right)  $ and $f\left(  -\infty \right)  $ which
belongs to $p^{-1}\left(  \mathcal{O}\right)  $ which by property
(\ref{superkarafinal}) is unique. Replace $f$ by this geodesic and, thus, we
may assume that $f\left(  +\infty \right)  =\beta \left(  +\infty \right)  .$
Similarly we arrange so that $h\left(  +\infty \right)  =\beta \left(
-\infty \right)  .$ Denote by $\phi$ the hyperbolic isometry corresponding to
$\beta.$

For each $n,$ by extending the geodesic segment joining 
$f\left(  0\right)$ with $\phi^{n}\left(  h\left( 0\right)  \right) $
to a geodesic line, it follows that the function 
$\alpha \bigl(\cdot ,f\left(  0\right)  ,\phi^{n}\left(  h\left(
0\right)  \right)  \bigr )$ 
attains positive and negative values on $\partial\widetilde{X} .$ 
As $\partial\widetilde{X} $ is assumed to be connected, there exists $\xi_{n}$ 
in $\partial \widetilde{X}$ such that
$\alpha \bigl(\xi_{n},f\left(  0\right)  ,\phi^{n}\left(  h\left(  0\right)
\right)  \bigr
)=0.$ We claim that
\begin{equation}
\xi_{n}\rightarrow f\left(  +\infty \right)  \, \,as\, \,n\rightarrow
\infty.\label{xn}%
\end{equation}
To see this assume, on the contrary, that $\left \{  \xi_{n}\right \}  $ (or, a
subsequence of it) converges to $\xi \in \partial X$ with $\xi \neq f\left(
+\infty \right)  .$ Let $M$ be a positive real number. 
For each fixed $n,$ using equation (\ref{generalizeddd}) and the fact that 
$\xi_{n}$ is chosen so that
$\phi^{n}\left(  h\left(  0\right)  \right)  $ and $f\left(  0\right)  $ are
equidistant from $\xi_{n},$ we may pick a sequence $\left \{
x^n_{m}\right \}  _{m\in \mathbb{N}}$ with the property $x^n_{m}\rightarrow
\xi_{n}$ as $m\rightarrow \infty$ and 
\[
\left \vert \alpha \bigl(x^n_{m},f\left(  0\right)  ,\phi^{n}\left(  h\left(
0\right)  \right)  \bigr )\right \vert \leq M,\, \, \, \, \mathrm{for}%
\, \, \mathrm{all}\, \, \,m\, \mathrm{large}\, \mathrm{enough.}%
\]
with $M>0$ being independent of $n.$ 
It is well known (see, for example, \cite[Ch. I, \S 4]{[Coo]}) that 
$\widetilde{X} \cup\partial \widetilde{X}$ is metrizable, hence, 
by a diagonal argument we obtain a
sequence $\left \{  x^n_{m\left(  n\right)}\right \}  _{n\in \mathbb{N}}$ such
that $x^n_{m\left(  n\right) }\rightarrow \xi$ as $n\rightarrow \infty$ and
\begin{equation}
\left \vert \alpha \bigl(x^n_{m\left(  n\right) },f\left(  0\right)  ,\phi
^{n}\left(  h\left(  0\right)  \right)  \bigr )\right \vert \leq M%
,\, \, \, \, \mathrm{for}\, \, \mathrm{all}\, \, \,n.\label{duoStar}%
\end{equation}
For the hyperbolic product of the sequences  $\left \{  x^n_{m\left(  n\right)}\right \}  _{n\in \mathbb{N}}$ and $\left \{ \phi^{n}\left(  h\left(  0\right) \right) \right \}  _{n\in \mathbb{N}}$ with 
base point $f(0)$ we have 
\begin{eqnarray*}
2\left(  x^n_{m\left(  n\right)} ,  \phi^{n}\left(  h\left(  0\right) \right) \right)_{f(0)} 
& = &
 d\left( f(0), x^n_{m\left(  n\right)} \right) 
+ d\left( f(0),  \phi^{n}\left(  h\left(  0\right) \right) \right)-
\\ 
 & &  \hskip38mm
 - d\left(   \phi^{n}\left(  h\left(  0\right) \right) , x^n_{m\left(  n\right)} \right)
 \\
 & = & - \alpha \left(x^n_{m\left(  n\right) },f\left(  0\right)  ,\phi
^{n}\left(  h\left(  0\right)  \right)  \right) + 
d\bigl (  f(0),  \phi^{n}\left(  h\left(  0\right) \right) \bigr)
\end{eqnarray*}
It follows by (\ref{duoStar}) that $\left(  x^n_{m\left(  n\right)} ,  \phi^{n}\left(  h\left(  0\right) \right) \right)_{f(0)} \rightarrow \infty$ as $n\rightarrow \infty,$ hence,  
the sequences $\left \{  x^n_{m\left(  n\right)}\right \}  _{n\in \mathbb{N}}$ and $\left \{ \phi^{n}\left(  h\left(  0\right) \right) \right \}  _{n\in \mathbb{N}}$ define the same point
at the boundary. This is a contradiction,  since  $\phi^{n}\left(  h\left(  0\right)  \right)  \rightarrow \beta \left( +\infty \right) = f\left( +\infty \right)$ and 
$\left \{  x^n_{m\left(  n\right)}\right \}  _{n\in \mathbb{N}} \rightarrow \xi .$

Thus equation (\ref{xn}) is proved. In a similar manner we show that
\begin{equation}
\phi^{-n}\left(  \xi_{n}\right)  \rightarrow h\left(  +\infty \right)
\, \,as\, \,n\rightarrow \infty.\label{yn}%
\end{equation}
Choose now geodesics $f_{n}\in G\widetilde{X},$ $n\in \mathbb{N}$ such that
$f_{n}\left(  +\infty \right)  =\xi_{n}$ and $f_{n}\left(  -\infty \right)
=f\left(  -\infty \right)  $ and by Lemma \ref{karafinalL}(b), we may
parametrize $f_{n}$ so that $f_{n}\rightarrow f$ or, equivalently,
$f_{n}\left(  0\right)  \rightarrow f\left(  0\right)  .$ Similarly, choose
$h_{n}\in GX$ such that $h_{n}\left(  +\infty \right)  =\xi_{n}$ and
$h_{n}\left(  -\infty \right)  =\phi^{n}\bigl
(h\left(  -\infty \right)  \bigr)$ and parametrize them so that
\begin{equation}
\alpha \bigl(\xi_{n},f_{n}\left(  0\right)  ,h_{n}\left(  0\right)
\bigr)=0.\label{finalb}%
\end{equation}
It is apparent that for $n$ large enough, $f_{n}\in p^{-1}\left(
\mathcal{O}\right)  $ and $h_{n}\in W^{ss}\left(  f_{n}\right)  .$ If we show
that $\phi^{-n}\left(  h_{n}\right)  \in p^{-1}\left(  \mathcal{U}\right)  $
for $n$ large enough, then we would have
\[%
\begin{array}
[c]{l}%
p\left(  f_{n}\right)  \in \mathcal{O},\\
p\left(  h_{n}\right)  =p\bigl(\phi^{-n}\left(  h_{n}\right)  \bigr)\in
\mathcal{U},\, \,and\\
p\left(  h_{n}\right)  \in W^{ss}\bigl(p\left(  f_{n}\right)  \bigr)
\end{array}
\]
The above three properties imply that for $n$ large enough, $W^{ss}\bigl
(p\left(  f_{n}\right)  \bigr)\cap \mathcal{U}\neq \emptyset,$ as required in
equation (\ref{basicoldstep}). We conclude the proof of the proposition by
showing that $\phi^{-n}\left(  h_{n}\right)  \in p^{-1}\left(  \mathcal{U}%
\right)  $ for $n$ large enough. In fact we will show that $\phi^{-n}%
h_{n}\rightarrow h.$ Clearly,
\begin{equation}
\bigl(\phi^{-n}\left(  h_{n}\right)  \bigr)\left(  +\infty \right)  =\phi
^{-n}\bigl(h_{n}\left(  \infty \right)  \bigr)=\phi^{-n}\left(  \xi_{n}\right)
\rightarrow h\left(  +\infty \right)  \label{a}%
\end{equation}
and
\begin{equation}
\bigl(\phi^{-n}\left(  h_{n}\right)  \bigr)\left(  -\infty \right)  =\phi
^{-n}\bigl(h_{n}\left(  -\infty \right)  \bigr)=h\left(  -\infty \right)
\label{c}%
\end{equation}
Use equations (\ref{a}), (\ref{c}) to apply Lemma \ref{karafinalL}(b) for the
unique geodesic $h$ to obtain a re-parametrization, say $\overline{h}_{n},$ of
each $\phi^{-n}\left(  h_{n}\right)  $ such that $\overline{h}_{n}\rightarrow
h.$ In particular, we have
\[
d\Bigl(h\left(  0\right)  ,\mathrm{Im}\overline{h}_{n}\Bigr)\rightarrow0
\]
which implies
\[
d\Bigl(h\left(  0\right)  ,\mathrm{Im}\, \phi^{-n}\left(  h_{n}\right)
\Bigr)\rightarrow0
\]
as $n\rightarrow+\infty.$ Therefore,
\begin{equation}
d\Bigl(\phi^{n}\bigl(h\left(  0\right)  \bigr),\mathrm{Im}\,h_{n}%
\Bigr)\rightarrow0\, \,as\, \,n\rightarrow \infty \label{finala}%
\end{equation}
Let $h_{n}\left(  t_{n}\right)  ,t_{n}\in \mathbb{R}$ be the point on
$\mathrm{Im}\,h_{n}$ which realizes the distance in equation (\ref{finala})
above. As the function $\alpha$ is Lipschitz with respect to the third
variable (with Lipschitz constant 1) we have
\[
\left \vert \alpha \bigl(\xi_{n},f\left(  0\right)  ,\phi^{n}\bigl(h\left(
0\right)  \bigr)\bigr)-\alpha \bigl(\xi_{n},f\left(  0\right)  ,h_{n}\left(
t_{n}\right)  \bigr)\right \vert \leq d\Bigl(\phi^{n}\bigl(h\left(  0\right)
\bigr),h_{n}\left(  t_{n}\right)  \Bigr)
\]
Using the defining property of $\xi_{n},$ i.e. $\alpha \bigl(\xi_{n},f\left(
0\right)  ,\phi^{n}\left(  h\left(  0\right)  \right)  \bigr)=0,$ it follows
that
\[
\alpha \bigl(\xi_{n},f\left(  0\right)  ,h_{n}\left(  t_{n}\right)  \bigr
)\rightarrow0\, \,as\, \,n\rightarrow \infty
\]
Similarly, using the fact that $f_{n}\left(  0\right)  \rightarrow f\left(
0\right)  $ as $n\rightarrow \infty$ and the Lipschitz property of $\alpha$
with respect to the second variable we have
\[
\alpha \bigl(\xi_{n},f_{n}\left(  0\right)  ,h_{n}\left(  t_{n}\right)  \bigr
)\rightarrow0\, \,as\, \,n\rightarrow \infty
\]
Since, by lemma \ref{busman}(c), there is a unique point on each
$\mathrm{Im}\,h_{n}$ which is equidistant from $f_{n}\left(  0\right)  $ with
respect to $\xi_{n},$ namely, $h_{n}\left(  0\right)  $ $\bigl($cf. equation
(\ref{finalb})$\bigr),$ it follows that $t_{n}\rightarrow0$ which, combined
with equation (\ref{finala}) implies that
\[
d\Bigl(\phi^{n}\bigl(h\left(  0\right)  \bigr),h_{n}\left(  0\right)
\Bigr)\rightarrow0\, \,as\, \,n\rightarrow \infty.
\]
Therefore, $\phi^{-n}\bigl(h_{n}\left(  0\right)  \bigr)\rightarrow h\left(
0\right)  $ as $n\rightarrow \infty.$ By proposition \ref{bakis} it follows
that $\phi^{-n}h_{n}\rightarrow h$ which implies that $\phi^{-n}\left(
h_{n}\right)  \in p^{-1}\left(  \mathcal{U}\right)  $ concluding the proof.
\end{proof}

\begin{proposition}
\label{ssclosure_closed}For every  closed geodesic $c\in GX ,$  $\overline
{W^{ss}\left(  c\right)  }=GX.$
\end{proposition}

\begin{proof}
Let $g$ be the geodesic produced in Proposition \ref{ssclosure_somegeodesic},
that is, $\overline{W^{ss}\left(  g\right)  }=GX$ and let $c$ be a closed
geodesic. By Proposition \ref{all_stable},
$\overline{W^{s}\left(  c\right)  }=GX$ so that $g\in \overline{W^{s}\left(
c\right)  }=GX.$ Thus, there exists a sequence $\left \{  g_{n}\right \}
\subset W^{s}\left(  c\right)  $ such that $g_{n}\rightarrow g.$ For each
$n\in \mathbb{N},$ consider lifts $\widetilde{g}_{n}$ and $\widetilde{c}$ of
$g_{n}$ and $c$ respectively satisfying $\widetilde{g}_{n}\in W^{s}\left(
\widetilde{c}\right)  $ and use Lemma \ref{busman}(c) to obtain a real number
$\widetilde{t}_{n}$ such that $\widetilde{t}_{n}\cdot g_{n}\in W^{ss}\left(
c\right)  .$ Each $\widetilde{t}_{n}$ may be written as
\[
\widetilde{t}_{n}=k\omega+t_{n}%
\]
where $k\in \mathbb{Z}$ and $t_{n}\in \left[  0,\omega \right)  .$ By choosing,
if necessary a subsequence, $t_{n}\rightarrow t$ for some $t\in \left[
0,\omega \right]  .$ Then $t_{n}\cdot g_{n}\rightarrow t\cdot g$ and
$t_{n}\cdot g_{n}\in W^{ss}\left(  c\right)  $ which simply means that $t\cdot
g\in \overline{W^{ss}\left(  c\right)  }$ and by Lemma \ref{contofsets}(c) we
have $\overline{W^{ss}\left(  c\right)  }\supset \overline{W^{ss}\left(  t\cdot
g\right)  }=\overline{t\cdot W^{ss}\left(  g\right)  }=GX.$
\end{proof}

We will need a point-wise version of topological mixing and a criterion for
such a property.

\begin{definition}
Let $h$ and $f$ be in $GX$ and let $\left \{  s_{n}\right \}  _{n\in \mathbb{N}}$
be a sequence converging to $+\infty$ or $-\infty.$ We say that $h$ is $s_{n}%
$-mixing with $f$ (notation, $h\sim_{s_{n}}f$) if for every neighborhood
$\mathcal{O}$ and $\mathcal{U}$ in $GX$ of $h$ and $f$ respectively,
$s_{n}\cdot \mathcal{O}\cap \mathcal{U}\neq \emptyset$ for all $n$ sufficiently large.
\end{definition}

For a geodesic $h,$ denote by $-h$ the geodesic with the reverse orientation,
that is, $\left(  -h\right)  \left(  t\right)  :=h\left(  -t\right)  ,$
$t\in \mathbb{R}.$ For a neighborhood $\mathcal{O}$ of $h$ denote by
$-\mathcal{O}$ the neighborhood of $-h$ defined by $-\mathcal{O}:=\left \{
-f\bigm \vert f\in \mathcal{O}\right \}  .$

\begin{lemma}
\label{snLemma}Let $\left \{  s_{n}\right \}  _{n\in \mathbb{N}}$ be a sequence
converging to $+\infty$ or $-\infty.$ Then
\[
h\sim_{s_{n}}f\Leftrightarrow f\sim_{-s_{n}}h\Leftrightarrow-f\sim_{s_{n}}-h.
\]

\end{lemma}

\begin{proof}
Let $\mathcal{O}$ and $\mathcal{U}$ in $GX$ be arbitrary neighborhoods of $h$
and $f$ respectively. The assumption $h\sim_{s_{n}}f$ means that for each
$n\in \mathbb{N},$ there exists a geodesic $h_{n}^{\prime}\in \mathcal{O}$ such
that $s_{n}\cdot h_{n}^{\prime}\in \mathcal{U}$ or, equivalently, $\left(
-s_{n}\right)  \cdot \left(  s_{n}\cdot h_{n}^{\prime}\right)  \in \left(
-s_{n}\right)  \cdot \mathcal{U}$. In other words, $h_{n}^{\prime}\in \left(
-s_{n}\right)  \cdot \mathcal{U}$ which implies that $\left(  -s_{n}\right)
\cdot \mathcal{U}\cap \mathcal{O}\neq \emptyset.$ This shows that $f\sim_{-s_{n}%
}h.$ The converse of the first equivalence is trivial as $\left \{  -\left(
-s_{n}\right)  \right \}  =\left \{  s_{n}\right \}  .$

Assuming $f\sim_{-s_{n}}h$ we have, by definition, that $\left(
-s_{n}\right)  \cdot \mathcal{U}\cap \mathcal{O}\neq \emptyset$ for all large
$n.$ Thus, for each large enough $n\in \mathbb{N},$ there exists a geodesic
$f_{n}^{\prime}\in \mathcal{U}$ such that $\left(  -s_{n}\right)  \cdot
f_{n}^{\prime}\in \mathcal{O}$ or, equivalently, $-\left[  \left(
-s_{n}\right)  \cdot f_{n}^{\prime}\right]  \in-\mathcal{O}.$ Since $-\left[
\left(  -s_{n}\right)  \cdot f_{n}^{\prime}\right]  =s_{n}\cdot \left(
-f_{n}^{\prime}\right)  $ we have $s_{n}\cdot \left(  -f_{n}^{\prime}\right)
\in-\mathcal{O}.$ Clearly, $-f_{n}^{\prime}\in-\mathcal{U},$ and thus,
$s_{n}\cdot \left(  -\mathcal{U}\right)  \cap \left(  -\mathcal{O}\right)
\neq \emptyset.$ This shows that $-f\sim_{s_{n}}-h.$ The proof of the converse
of the second equivalence is again trivial as $-\left(  -f\right)  =f.$
\end{proof}

\begin{remark}
\label{snRemark}The first equivalence in the previous Lemma shows that the
$s_{n}$-mixing relation is not a symmetric relation. In particular, it is not
an equivalence relation.
\end{remark}

The following criterion for the $s_{n}$-mixing of $h,f$ holds.

\begin{lemma}
\label{criterion}If $h$ and $f\in GX$, then $h\sim_{s_{n}}f$ if and only if
for each subsequence $\left \{  s_{n}^{\prime}\right \}  $ of $\left \{
s_{n}\right \}  $ there exists a subsequence $\left \{  r_{n}\right \}  $ of
$\left \{  s_{n}^{\prime}\right \}  $ and a sequence of non-closed geodesics
$\left \{  h_{n}\right \}  \subset GX$ such that $h_{n}\rightarrow h,r_{n}\cdot
h_{n}\rightarrow f$ and $\widetilde{h}_{n}\left(  +\infty \right)  \notin
F_{h}^{nu}$ for some, hence any, lift $\widetilde{h}_{n}$ of $h_{n}.$
\end{lemma}

\begin{proof}
If $h\sim_{s_{n}}f$ for some $h$ and $f\in GX,$ then using decreasing
sequences of open neighborhoods of $h$ and $f$ it is easily shown that for
each subsequence $\left \{  s_{n}^{\prime}\right \}  $ of $\left \{
s_{n}\right \}  $ there exists a subsequence $\left \{  r_{n}\right \}  $ of
$\left \{  s_{n}^{\prime}\right \}  $ and a sequence $\left \{  h_{n}\right \}
\subset GX$ such that $h_{n}\rightarrow h$ and $r_{n}\cdot h_{n}\rightarrow
f.$ We proceed to show that we may replace $\left \{  h_{n}\right \}  $ by a
sequence $\left \{  g_{n}\right \}  $ of non-closed geodesics so that
$g_{n}\rightarrow h$ and $r_{n}\cdot g_{n}\rightarrow f.$ \newline Let
$\gamma$ be the geodesic posited in Theorem \ref{orbit}, that is, its orbit
$\mathbb{R}\cdot \gamma$ is dense in $GX.$ As observed at the end of the proof
of Theorem \ref{orbit}, $\gamma$ is non-closed. Thus, there exists a sequence
$\left \{  t_{i}^{1}\right \}  _{i\in \mathbb{N}}$ such that $t_{i}^{1}%
\cdot \gamma \rightarrow h_{1}.$ Set $g_{i}^{1}=t_{i}^{1}\cdot \gamma$ and,
clearly, all $g_{i}^{1}$ are non-closed. Similarly, for each $h_{n},$ we may
find a sequence of non-closed geodesics $g_{i}^{n}=t_{i}^{n}\cdot \gamma$
converging to $h_{n}.$ By a diagonal argument we obtain a sequence of
non-closed geodesics $\left \{  g_{n}\right \}  $ converging to $h$ and,
clearly, $\mathrm{lim}_{n\rightarrow \infty}r_{n}\cdot g_{n}=\mathrm{lim}%
_{n\rightarrow \infty}r_{n}\cdot h_{n}=f.$

By Lemma \ref{denseorbitunc}, the geodesic $\gamma$ having dense orbit can be
chosen so that $\widetilde{\gamma}\left(  +\infty \right)  \notin F_{h}^{nu}.$
Since the non-closed geodesics $\left \{  g_{n}\right \}  $ constructed above
are all translates of $\gamma$ the last requirement of the Lemma is fulfilled.

The proof of the converse statement is elementary.
\end{proof}

\begin{remark}
\label{suggestedrem2}For a geodesic $f\in GX$ and a sequence $s_{n}%
\rightarrow \infty$ the set
\[
\left \{  h\in GX:h\sim_{s_{n}}f\right \}
\]
is a closed set.
\end{remark}

\begin{proof}
Assume $\left \{  h_{k} \right \}  $ is a sequence with $h_{k} \rightarrow h$
and $h_{k} \sim_{s_{n}}f$ for all $k\in \mathbb{N}.$ We show that $h\sim
_{s_{n}}f .$ Let $\mathcal{O}$ and $\mathcal{U}$ be arbitrary neighborhoods of
$h$ and $f$ respectively. Find $k_{0}$ such that $h_{k_{0}} \in \mathcal{O} . $
Then, as $h_{k_{0}} \sim_{s_{n}}f ,$ we have $s_{n}\cdot \mathcal{O}%
\cap \mathcal{U}\neq \emptyset$ for all $n$ sufficiently large. The latter
means, by definition, that $h\sim_{s_{n}}f .$
\end{proof}

We next show the following lemma which asserts that point-wise topological
mixing is transferred via the strong stable relation of geodesics.

\begin{lemma}
\label{new}Let $f,g$ and $g^{\prime}\in GX$ so that $f\in \overline
{W^{ss}\left(  g\right)  },f$ is non-closed, $g$ is closed
and unique with $  \widetilde{g}\left(  +\infty \right) \in F_{h}^{u}$ for some,
hence any, lift $\widetilde{g}$  of $g.$ Then, if $g\sim_{s_{n}}g^{\prime}$  
for some sequence $s_{n}\rightarrow \infty,$ then $f\sim_{s_{n}}g^{\prime}.$
\end{lemma}

\begin{proof}
Fix a sequence $s_{n}\rightarrow \infty.$ By Remark \ref{suggestedrem2}, it
suffices to prove the assertion of the lemma for $f$ non-closed and $f\in
W^{ss}\left(  g\right)  .$ The rest of the proof follows the line of proof
given in \cite[Lemma 4.4]{[CT3]} which we include here since several
restrictions apply in our setup.

In order to use Lemma \ref{criterion} above for showing that $f\sim_{s_{n}%
}g^{\prime},$ let $\left \{  t_{n}\right \}  $ be arbitrary subsequence of
$\left \{  s_{n}\right \}  .$ As $g\sim_{s_{n}}g^{\prime}$ there exists (again
by Lemma \ref{criterion}) a subsequence $\left \{  r_{n}\right \}  $ of
$\left \{  t_{n}\right \}  $ and a sequence $\left \{  g_{n}\right \}  $ such
that
\[
g_{n}\rightarrow g\, \, \mathrm{and}\, \,r_{n}\cdot g_{n}\rightarrow
g^{\prime}.
\]
Lift $g$ and $f$ to geodesics $\overline{g}$ and $\overline{f}$ in
$G\widetilde{X}$ such that $\overline{f}\left(  +\infty \right)  =\overline
{g}\left(  +\infty \right)  $ and 
\[
\alpha \left(  \overline{f}\left(
+\infty \right)  ,\overline{f}\left(  0\right)  ,\overline{g}\left(  0\right)
\right)  =0
\]
Lift each $g_{n}$ to a geodesic $\overline{g_{n}}$ such that
$\overline{g_{n}}\rightarrow \overline{g}.$ Since $g$ is unique, the latter is equivalent
to $\overline{g_{n}}\left(  +\infty \right)  \rightarrow \overline{g}\left(
+\infty \right)  ,$ $\overline{g_{n}}\left(  -\infty \right)  \rightarrow
\overline{g}\left(  -\infty \right)  $ and $\overline{g_{n}}\left(  0\right)
\rightarrow \overline{g}\left(  0\right)  .$ We may assume (cf Lemma
\ref{criterion}) that
\[
\left \{  \overline{g_{n}}\left(  +\infty \right)  \bigm \vert n\in
\mathbb{N}\right \}  \subset \partial \widetilde{X}\setminus F_{h}^{nu}.
\]
Use Lemma \ref{karafinalL}(b) to define a sequence of geodesics $\left \{
\overline{f_{n}}\right \}  _{n\in \mathbb{N}}$ such that $\overline{f_{n}%
}\rightarrow \overline{f}$ with $\overline{f_{n}}\left(  +\infty \right)
=\overline{g_{n}}\left(  +\infty \right)  $ and $\overline{f_{n}}\left(
-\infty \right)  =\overline{f}\left(  -\infty \right)  .$ By the continuity of
the $\alpha$ function we have that
\[
\mathrm{lim}_{n\rightarrow \infty}\alpha \bigl(\xi_{n},\overline{f_{n}}\left(
0\right)  ,\overline{g_{n}}\left(  0\right)  \bigr)=\alpha \bigl(\xi
,\overline{f}\left(  0\right)  ,\overline{g}\left(  0\right)  \bigr)=0
\]
hence, by passing if necessary to a subsequence of $\left \{  \overline{f_{n}%
}\right \}  _{n\in \mathbb{N}},$ we may assume that
\[
\alpha \bigl(\xi_{n},\overline{f_{n}}\left(  0\right)  ,\overline{g_{n}}\left(
0\right)  \bigr)<1/n,\, \, \,for\, \,all\, \,n\in \mathbb{N}%
\]
By lemma \ref{busman}(c) we may choose the parametrization of each
$\overline{f_{n}}$ so that
\begin{equation}
\alpha \bigl(\xi_{n},\overline{f_{n}}\left(  0\right)  ,\overline{g_{n}}\left(
0\right)  \bigr)=0,\, \, \,for\, \,all\, \,n\in \mathbb{N} \label{efen}%
\end{equation}
As the change of parametrization tends to $0$ as $n\rightarrow \infty$ we may
assume that the sequence $\left \{  \overline{f_{n}}\right \}  _{n\in \mathbb{N}%
}$ satisfies equation (\ref{efen}) and $\overline{f_{n}}\rightarrow
\overline{f}.$ Moreover, if we set $\,f_{n}:=p\left(  \overline{f_{n}}\right)
$ then $f_{n}\rightarrow f.$ We proceed now to show that $r_{n}\cdot
f_{n}\rightarrow g^{\prime}.$ Let $K$ be an arbitrary compact subset of
$\mathbb{R}$ and $\varepsilon$ arbitrary positive. By construction,
$\overline{f_{n}}\in W^{ss}\left(  \overline{g_{n}}\right)  $ for all
$n\in \mathbb{N}$ and $\overline{f}\in W^{ss}\left(  \overline{g}\right)  .$
Since $\overline{g_{n}}\left(  +\infty \right)  \in \partial \widetilde
{X}\setminus F_{h}^{nu},$ condition (C) applies for all pairs $\overline{g_{n}%
},\overline{f_{n}}$ and $\overline{g},\overline{f}.$ Therefore, by Proposition
\ref{asy},
\begin{equation}%
\begin{array}
[c]{l}%
\mathrm{lim}_{t\rightarrow \infty}d\bigl(\overline{f_{n}}\left(  t\right)
,\overline{g_{n}}\left(  t\right)  \bigr)=0,\, \, \, \mathrm{and}\\
\mathrm{lim}_{t\rightarrow \infty}d\bigl(\overline{f}\left(  t\right)
,\overline{g}\left(  t\right)  \bigr)=0.
\end{array}
\label{double}%
\end{equation}
Choose a positive real $T$ such that
\[
d\bigl(\overline{f}\left(  T\right)  ,\overline{g}\left(  T\right)  \bigr
)<\varepsilon/6
\]
The above equation holds for all $t>T.$ This follows by convexity of the
distance function (see \cite[Ch.I, Proposition 5.4]{[Bal]}) and equation
(\ref{double}). As $\overline{f_{n}}\rightarrow \overline{f}$ and
$\overline{g_{n}}\rightarrow \overline{g}$ we may choose $N\in \mathbb{N}$ such
that for all $n\geq N$
\[%
\begin{array}
[c]{l}%
d\bigl(\overline{f_{n}}\left(  T\right)  ,\overline{f}\left(  T\right)
\bigr
)<\varepsilon/6,\, \, \, \mathrm{and}\\
d\bigl(\overline{g_{n}}\left(  T\right)  ,\overline{g}\left(  T\right)
\bigr )<\varepsilon/6.
\end{array}
\]
Thus, $d\bigl(\overline{f_{n}}\left(  T\right)  ,\overline{g_{n}}\left(
T\right)  \bigr
)<\varepsilon/2$ and as before, it follows that
\[
d\bigl(\overline{f_{n}}\left(  t\right)  ,\overline{g_{n}}\left(  t\right)
\bigr)<\varepsilon/2\, \, \, \mathrm{for\, \, \,all\, \,}t>T.
\]
As $r_{n}\rightarrow+\infty,$ there exists $n_{0}$ such that $r_{n}\geq
T- \min K$ for all $n\geq n_{0}.$ Now for all $n$ sufficiently large, namely,
$n\geq \mathrm{max}\left \{  N,n_{0}\right \}  ,$ we have
\[
d\bigl(\overline{f_{n}}\left(  r_{n}+t\right)  ,\overline{g_{n}}\left(
r_{n}+t\right)  \bigr)<\varepsilon/2,\, \, \forall \, \,t\in K
\]
which implies that
\[
d\bigl(r_{n}\cdot f_{n}\left(  t\right)  ,r_{n}\cdot g_{n}\left(  t\right)
\bigr)<\varepsilon/2,\, \, \forall \, \,t\in K
\]
As $r_{n}\cdot g_{n}\rightarrow g^{\prime},$ we have that for all $n$
sufficiently large
\[
d\bigl(r_{n}\cdot g_{n}\left(  t\right)  ,g^{\prime}\left(  t\right)
\bigr)<\varepsilon/2,\, \, \forall \, \,t\in K
\]
Combining the last two inequalities we obtain that
\[
d\bigl(r_{n}\cdot f_{n}\left(  t\right)  ,g^{\prime}\left(  t\right)
\bigr)<\varepsilon,\, \, \forall \, \,t\in K
\]
As $K,\varepsilon$ were arbitrary, we have shown that for all $n$ sufficiently
large, $r_{n}\cdot f_{n}$ lies in any neighborhood of $g^{\prime}.$ Therefore,
$r_{n}\cdot f_{n}\rightarrow g^{\prime}$ as required.
\end{proof}

\begin{proof}
[Proof of Theorem \ref{mainn}]It suffices to show that
\begin{multline*}
\forall \,h,f\in GX\, \, \mathrm{and}\, \, \forall \, \left \{  t_{n}\right \}  \,
\mathrm{with}\,t_{n}\rightarrow \infty,\\
\exists \, \, \mathrm{sub}\text{-} \mathrm{sequence}\, \, \left \{
s_{n}\right \}  \subset \left \{  t_{n}\right \}  \, \, \mathrm{such}\, \,
\mathrm{that}\, \,h\sim_{s_{n}}f.
\end{multline*}
To see that this property is sufficient, assume it holds and, on the contrary,
the geodesic flow is not mixing. Then, there would exist neighborhoods
$\mathcal{O}$ and $\mathcal{U}$ in $GX$ such that: for each $n\in \mathbb{N},$
there exists $T_{n}>n$ so that $T_{n}\cdot \mathcal{O}\cap \mathcal{U}%
=\emptyset.$ Clearly, for any subsequence $\left \{  s_{n}\right \}  $ of
$\left \{  T_{n}\right \}  $ we have $s_{n}\cdot \mathcal{O}\cap \mathcal{U}%
=\emptyset.$ Thus, for any $h\in \mathcal{O}$ and $f\in \mathcal{U}$ the above
property does not hold.

Since the notion of $s_{n}$-mixing is defined via neighborhoods it suffices to
show the above property only for geodesics $h,f$ which are not closed$.$ This
will allow the use of Lemma \ref{new}. \newline By condition (D), let $c$ be a
closed and unique geodesic with
\[
\left \{  \widetilde{c}\left(  -\infty \right)  ,\widetilde{c}\left(
+\infty \right)  \right \}  \subset F_{h}^{u} \subset \partial \widetilde{X}
\setminus F_{h}^{nu}%
\]
for some, hence any, lift $\widetilde{c}$ of $c .$ By Proposition
\ref{ssclosure_closed} we have $\overline{W^{ss}\left(  c\right)  }=GX.$
Clearly, for all $t\in \mathbb{R},\overline{W^{ss}\left(  t\cdot c\right)
}=GX.$ Let $\left \{  s_{n}\right \}  $ be a subsequence of $\left \{
t_{n}\right \}  $ such that $s_{n}\cdot c\rightarrow t\cdot c$ for some
$t\in \left[  0,\omega \right]  $ where $\omega$ is the period of $c.$ Clearly,
for any neighborhood $\mathcal{U}$ of $t\cdot c,$ $s_{n}\cdot c\in \mathcal{U}$
for large enough $n\in \mathbb{N}.$ In other words, $c\sim_{s_{n}}t\cdot c.$ As
$f\in GX=\overline{W^{ss}\left(  c\right)  }$ is non-closed we may apply Lemma
\ref{new} to the geodesics $f,c$ and $t\cdot c$ to obtain $f\sim_{s_{n}}t\cdot
c.$ The latter is, by Lemma \ref{snLemma}, equivalent to $-t\cdot c\sim
_{s_{n}}-f.$ As $-h\in \overline{W^{ss}\left(  -t\cdot c\right)  }=GX$ we apply
Lemma \ref{new} to the geodesics $-h,-t\cdot c$ and $-f$ to obtain
$-h\sim_{s_{n}}-f.$ By Lemma \ref{snLemma}, the latter is equivalent to
$f\sim_{s_{n}}h,$ as required.
\end{proof}

\section{Examples and Applications}

\subsection{Euclidean Surfaces and their properties}

We start by recalling the notion of a Euclidean surface with conical singularities.

Let $S$ be a surface equipped with a Euclidean metric with finitely many
\textit{conical singularities} (or \textit{conical points}), say
$s_{1},...,s_{n}$ in its interior. Every point which is not conical is called
a \textit{regular point} of $S.$ Denote by $\theta(s_{i})$ the angle at each
$s_{i}$ and we assume that $\theta(s_{i})\in(2\pi,+\infty).$

We write $C\left(  v,\theta \right)  $ for the standard cone with vertex $v$
and cone angle $\theta,$ namely, $C\left(  v,\theta \right)  $ is the set
$\left \{  \left(  r,t\right)  :0\leq r,t\in \mathbb{R}/\theta \mathbb{Z}%
\right \}  $ equipped with the metric $ds^{2}=dr^{2}+r^{2}dt^{2}.$

\begin{definition}
A Euclidean surface with conical singularities $s_{1},...,s_{n}$ is a surface
$S$ equipped with a length metric $d\left(  \cdot,\cdot \right)  $ such that

\begin{itemize}
\item Every point $p\in S\setminus \left \{  s_{1},...,s_{n}\right \}  $ has a
neighborhood isometric to a disk or half disk in the Euclidean plane

\item Each $s_{i}\in \left \{  s_{1},...,s_{n}\right \}  \subset S\setminus
\partial S$ has a neighborhood isometric to a neighborhood of the vertex $v$
of the standard cone $C\left(  v,\theta \left(  s_{i}\right)  \right)  .$
\end{itemize}
\end{definition}

Clearly, the metric on $S$ is a length metric and the surface $S\ $will be
written $e.s.c.s.$ for brevity. Note that for genus $g\geq2,$ such Euclidean
structures exist, see \cite{[Tro]}. Let $\widetilde{S}$ be the universal
covering of $S$ and let $p:\widetilde{S}\rightarrow S$ be the covering
projection. Obviously, the universal covering $\widetilde{S}$ is homeomorphic
to $\mathbb{R}^{2}$ and by requiring $p$ to be a local isometric map we may
lift $d$ to a metric on $\widetilde{S},$ denoted again by $d,$ so that
$(\widetilde{S},$ $d)$ becomes a $e.s.c.s.$

We will use the following

\begin{theorem}
[{Theorem 12 in \cite{[CPT]}}]\label{mathZ} Let $g$ be a non-closed geodesic
or geodesic ray in a closed $e.s.c.s.$ $S$ with genus $\geq2.$ Then $d\left(
\operatorname{Im}g,\left \{  s_{1},...,s_{n}\right \}  \right)  =0.$
\end{theorem}

\begin{corollary}
\label{mathZC}Let $Q$ be a compact $e.s.c.s.$ and $g$ be a non-closed geodesic
or geodesic ray in $Q\setminus \partial Q.$ Then $d\left(  \operatorname{Im}%
g,\left \{  s_{1},...,s_{n}\right \}  \right)  =0.$
\end{corollary}

\begin{proof}
Let $Q^{+}$ be a copy of $Q$ and glue $Q$ and $Q^{+}$ along their boundaries
to obtain a closed surface $S$ with $2n$ conical singularities $s_{1}%
,...,s_{n},s_{1}^{+},...,s_{n}^{+}.$ By Theorem \ref{mathZ},
\[
d\left(  \operatorname{Im}g,\left \{  s_{1},...,s_{n},s_{1}^{+},...,s_{n}%
^{+}\right \}  \right)  =0.
\]
Since $\operatorname{Im}g\subset Q,$ it is clear that $d\left(
\operatorname{Im}g,\left \{  s_{1}^{+},...,s_{n}^{+}\right \}  \right)  >0,$
hence,
\[
d\left(  \operatorname{Im}g,\left \{  s_{1},...,s_{n}\right \}  \right)  =0.
\]
$\; \;$ \hfill
\end{proof}

\subsection{Examples of CAT(0) surfaces}

\label{tesduo} We give an example of a $2-$dimensional CAT(0) surface $X$ and
we will show that it satisfies all four assumptions of Theorem \ref{mainn}.
\newline Let $M$ be a finite area surface of genus $g\geq2$ with pinched
negative curvature. Let $c_{M}$ be a simple closed separating geodesic in $M$
such that the closure of at least one of the components of $M\setminus
\operatorname{Im}c_{M}$ is compact. Denote by $M_{1}$ the compact subsurface
of $M$ and by $M_{2}$ the closure of the other component. Observe that $M_{2}$
may contain finitely many cusps, hence, $M_{2}$ may not be compact. Clearly,
$M=M_{1}\cup_{\operatorname{Im}c_{M}}M_{2}.$

Consider a compact $e.s.c.s.$ $S_{1}$ of the same topological type as $M_{1}$
and with its boundary component $\partial S_{1}$ isometric to
$\operatorname{Im}c_{M} .$ Set
\begin{equation}
X=S_{1}\cup_{\operatorname{Im}c_{M}}M_{2} \label{newsurfaces}%
\end{equation}
to be the surface obtained by gluing $S_{1}$ with $M_{2}$ along their
boundaries. Such a surface $X$ is a CAT(0) space and \cite[Theorem
11.6]{[Ric]} applies to prove mixing provided that the Bowen--Margulis measure
on $\partial \widetilde{X}$ is finite. However, the class of surfaces as
defined by (\ref{newsurfaces}) includes examples where the Bowen--Margulis
measure is not finite. For example, take the surface $M$ to be the surface
with one cuspidal end constructed in \cite[Theorem 1.2]{[Dal]} whose
fundamental group is exotic and convergent, thus, the corresponding
Bowen--Margulis measure is infinite.

The above construction can be performed by using, instead of a single simple
closed separating geodesic, a collection $c_{1} , \ldots, c_{k}$ of pairwise
disjoint simple closed geodesics such that the union $\mathrm{Im}c_{1}
\cup \cdots \cup \mathrm{Im}c_{k}$ splits $M$ into two components with the
closure of at least one of them being compact.

The rest of this sub-section is devoted into showing that a surface $X$ as
defined in (\ref{newsurfaces}) satisfies all four assumption of Theorem
\ref{mainn} and, thus, establish the following

\begin{application}
If $X$ is a surface of the form $X=S_{1}\cup_{\operatorname{Im}c_{M}}M_{2}$
constructed in (\ref{newsurfaces}) above, then the geodesic flow
$\mathbb{R}\times GX\rightarrow GX$ is topologically mixing.
\end{application}

\noindent We start with Condition ($\mathrm{\Delta}$) by showing

\begin{proposition}
\label{hyperlast} The space $\widetilde{X}$ is hyperbolic in the sense of Gromov.
\end{proposition}

We need the following

\begin{lemma}
\label{segmentlemma}Let $\left[  x,y\right]  $ and $\left[  x,z\right]  $ be
geodesic segments in a CAT(0) geodesic metric space $Y$ with $d\left(
y,z\right)  \leq C_{0}$ for some $C_{0} >0.$ Then for every point $A\in \left[
x,y\right]  $ there exists $B\in \left[  x,z\right]  $ such that $d\left(
A,B\right)  \leq  C_{0}.$
\end{lemma}

\begin{proof}
The desired property holds for triangles in $\mathbb{R}^2$ hence, by CAT(0)
inequality, the proof follows. 
\end{proof}\\[4mm]

\begin{proof}
[Proof of Proposition \ref{hyperlast}]
If $X$ does not have any cusps, $X$ is closed and $\pi_1 (X)$ hyperbolic, in which case
we have nothing to show. Assume now that $X$ contains at least one cusp. 
Let $X_{0}$ be the subsurface of $X$
obtained as follows: consider a horoball based at each cusp of $X$ and remove
its interior. Then, $X_{0}$ is a compact surface with as many geodesic
boundary components as the number of cusps in $X.$ We may assume that the
length of all boundary components is bounded by $C_{0}>0.$

Clearly, the universal cover $\widetilde{X}_{0}$ of $X_{0}$ is a subsurface of
$\widetilde{X}$ whose fundamental domain is a polygon which can be obtained
from the ideal fundamental domain of $\widetilde{X}$ by cutting off all its
ideal vertices by horocycles. As $X_{0}$ is compact, hence hyperbolic, and the
fundamental group $\pi_{1}\left(  X_{0}\right)  =\pi_{1}\left(  X\right)  $ is
free, it follows that $\widetilde{X}_{0}$ is $\delta_{\widetilde{X}_{0}}%
$-hyperbolic for some $\delta_{\widetilde{X}_{0}}>0.$ We will use the
hyperbolicity of $\widetilde{X}_{0}$ to prove Proposition \ref{hyperlast}%
.\\[4mm]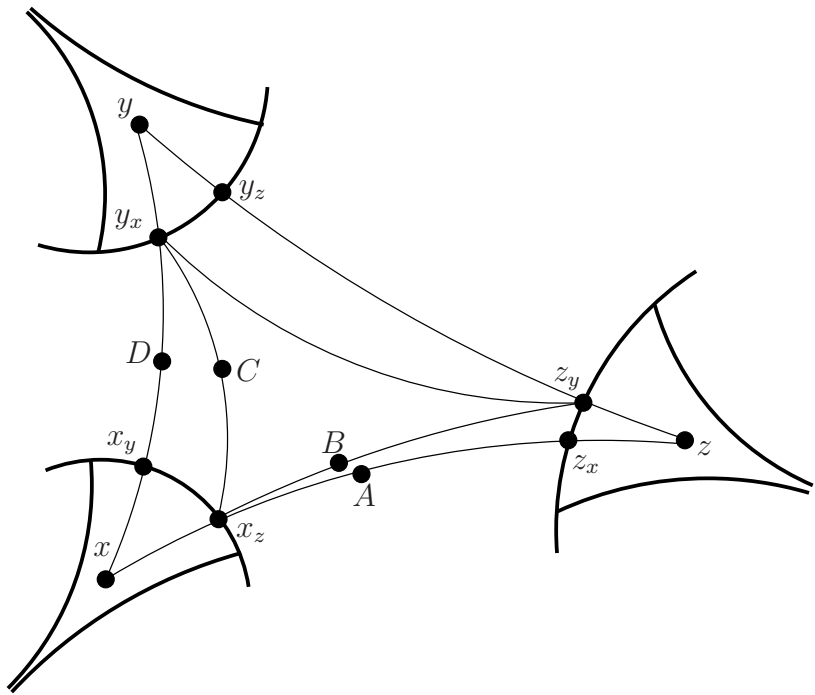
\begin{figure}[ptb]
\begin{center}
\includegraphics[scale=0.5]{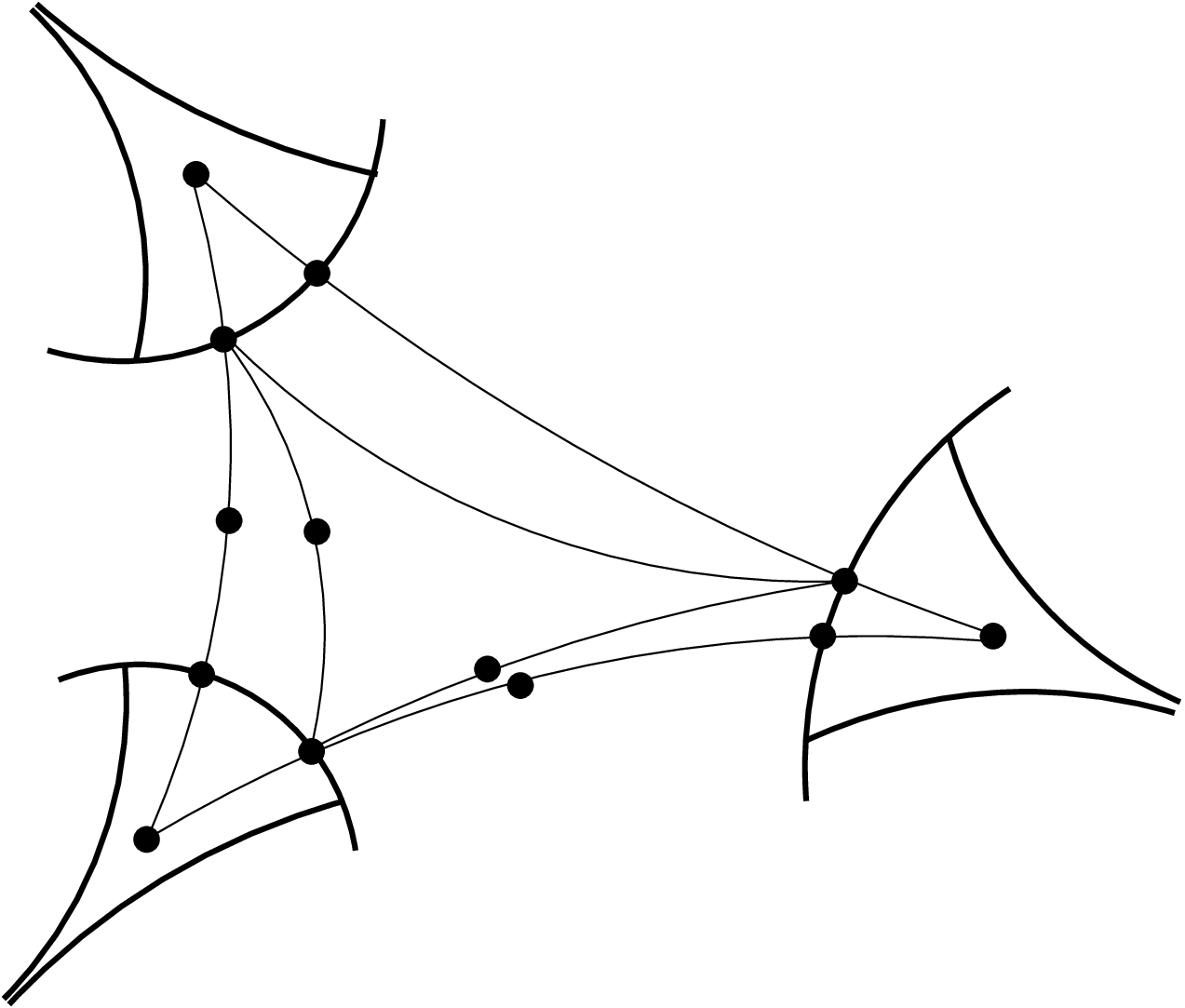}\\
\begin{picture}(22,12)
\put(-104,109){$x_y$}
\put(-55,74){$x_z$}
\put(-109,67){$x$}
\put(-101,193){$y_x$}
\put(-54,204){$y_z$}
\put(-100,235){$y$}
\put(72,100){$z_x$}
\put(65,133){$z_y$}
\put(119,105){$z$}
\put(-11,86){$A$}
\put(-23,106){$B$}
\put(-55,133){$C$}
\put(-97,140){$D$}
\end{picture}
\end{center}
\caption{The triangle $(x,y,x)$ in $\widetilde{X}$ and the thin triangle $(
x_{z} , z_{y} , y_{x} )$ in $\widetilde{X}_{0}$}%
\label{sximataki}%
\end{figure}\newline Let $\left(  x,y,z\right)  $ be a geodesic triangle in
$\widetilde{X}.$ We will show that every point $A\in \left[  x,z\right]  $
satisfies
\begin{equation}
d\left(  A,\left[  x,y\right]  \cup \left[  y,z\right]  \right)  \leq
\delta_{\widetilde{X}_{0}}+2C_{0} \label{deltainequality}%
\end{equation}
thus, showing that $\widetilde{X}$ is hyperbolic in the sense of Gromov.

Clearly, since $\widetilde{X}_{0}$ is $\delta_{\widetilde{X}_{0}}
-$hyperbolic, if $x,y,z\in \widetilde{X}_{0}$ we have nothing to show. We treat
the case $x,y,z\in \widetilde{X}\setminus \widetilde{X}_{0},$ and the other
cases can be treated similarly.

Denote by $\left[  x_{y},y_{x} \right] $ the intersection$\left[  x,y\right]
\cap \partial \widetilde{X}_{0}$ and, similarly, $\left[  x_{z},z_{x}\right] $
and $\left[  y_{z},z_{y}\right]  $ (see Figure \ref{sximataki}). Note that
both points $x_{y},x_{z}$ (resp. $y_{x},y_{z}$ and $z_{x},z_{y}$) belong to a
single horocycle side of $\widetilde{X}_{0}.$ Since the length of the boundary
components of $X_{0}$ are assumed to be bounded by $C_{0}$ we have
\begin{equation}
d\left(  x_{y},x_{z}\right)  \leq C_{0},d\left(  z_{x},z_{y}\right)  \leq
C_{0}~\mathrm{and}~d\left(  y_{x},y_{z}\right)  \leq C_{0}. \label{cmiden}%
\end{equation}
If $A\in \left[  x,x_{z}\right]  $ then, since $d\left(  x_{z},x_{y}\right)
\leq C_{0},$ by Lemma \ref{segmentlemma} we have
\[
d\left(  A,\left[  x,x_{y}\right]  \right)  \leq C_{0}.
\]
Similarly, if $A\in \left[  z_{x},z\right]  $ we obtain $d\left(  A,\left[
z,z_{y} \right]  \right)  \leq C_{0}.$
It follows that if $A\in \left[  x,x_{y}\right]  \cup \left[  z_{x},z\right]  $
then
\[
d\left(  A,\left[  x,y\right]  \cup \left[  y,z\right]  \right)  \leq C_{0}
\]
hence, the desired inequality (\ref{deltainequality}) holds.

If $A\in \left[  x_{z},z_{x}\right]  ,$ apply Lemma \ref{segmentlemma} to the
segments $\left[  x_{z},z_{x}\right]  $ and $\left[  x_{z},z_{y}\right]  $
using (\ref{cmiden}) to obtain a point
\[
B\in \left[  x_{z},z_{y}\right]  ~~\mathrm{with}~~d\left(  A,B\right)
\leq C_{0}.
\]
As $\widetilde{X}_{0}$ is $\delta_{\widetilde{X}_{0}} -$hyperbolic there
exists a point%
\[
C\in \left[  x_{z},y_{x}\right]  \cup \left[  y_{x},z_{y}\right]
~~\mathrm{with}~~d\left(  B,C\right)  \leq \delta_{\widetilde{X}_{0}}.
\]
Without loss of generality we may assume that $C\in \left[  x_{z},y_{x}\right]
.$ Again by Lemma \ref{segmentlemma} applied to the segments $\left[
x_{z},y_{x}\right]  $ and $\left[  x_{y},y_{x}\right]  $ we find a point
\[
D\in \left[  x_{y},y_{x}\right]  ~~\mathrm{with}~~d\left(  C,D\right)
\leq C_{0}.
\]
Combining the last three inequalities we obtain $d\left(  A,D\right)
\leq \delta_{\widetilde{X}_{0}}+2C_{0},$ thus, we have shown
(\ref{deltainequality}) for all $A\in \left[  x,y \right]  .$.
\end{proof}

\begin{proposition}
$\widetilde{X}$ satisfies Condition $(U)$.
\end{proposition}

\begin{proof}
Assume not, that is, assume there exist non-closed geodesics $f,f^{\prime}$ in
$G\widetilde{X}$ with $f\left(  +\infty \right)  =f^{\prime}\left(
+\infty \right)  $ and $f\left(  -\infty \right)  =f^{\prime}\left(
-\infty \right)  $ so that $\operatorname{Im}f\neq \operatorname{Im}f^{\prime}.$
Then by the Flat Strip Theorem, $\operatorname{Im}f$ and $\operatorname{Im}%
f^{\prime}$ bound a flat strip. Pick any geodesic $g$ in the interior of the
flat strip. Clearly, $\operatorname{Im}g$ has positive distance from the set
of conical points in $\widetilde{X}$ and $p\left(  g\right)  $ does not
intersect $M_{2}.$ That is, $p\left(  g\right)  $ in contained in the compact
$e.s.c.s.$ $S_{1}$ and
\[
d\left(  \operatorname{Im}p\left(  g\right)  ,\left \{  s_{1},...,s_{n}%
\right \}  \right)  >0.
\]
Since $g$ is homotopic to both $f,f^{\prime}$ it projects to a non-closed
geodesic, thus, the above inequality contradicts Corollary \ref{mathZC}.
\end{proof}

We proceed now to show Condition ($\mathrm{C}$).

\begin{proposition}
\label{condU}Let $g_{1},g_{2}$ be asymptotic geodesics with
\[
\xi=g_{1}\left(  +\infty \right)  =g_{2}\left(  +\infty \right)  \in
\partial \widetilde{X}\setminus F_{h}^{nu}.
\]
Then for appropriate parametrizations of $g_{1},g_{2}$ we have
\[
\lim_{t\rightarrow \infty}d\bigl(g_{1}\left(  t\right)  ,g_{2}\left(  t\right)
\bigr)=0.
\]

\end{proposition}

We first show the following

\begin{lemma}
\label{karaneww}Let $g_{1},g_{2}$ be two geodesics with
\[
\xi=g_{1}\left(  +\infty \right)  =g_{2}\left(  +\infty \right)  \in
\partial \widetilde{X}\setminus F_{h}^{nu}%
\]
all as in the above Proposition. Assume that
\[
d\left(  \mathrm{Im}g_{1},\mathrm{Im}g_{2}\right)  :=\inf \left \{  d\left(
x,y\right)  \bigm \vert x\in \mathrm{Im}g_{1},y\in \mathrm{Im}g_{2}\right \}
=0.
\]
Then there exists a unique re-parametrization $\overline{g}_{1}$ of $g_{1}$
such that
\[
\lim_{t\rightarrow \infty}d\bigl(\overline{g}_{1}\left(  t\right)
,g_{2}\left(  t\right)  \bigr)=0.
\]

\end{lemma}

\begin{proof}
As $g_{1},g_{2}$ are asymptotic, the distance function
\[
t\rightarrow d\left(  g_{1}\left(  t\right)  ,g_{2}\left(  t\right)  \right)
,t\geq0
\]
is convex (see \cite{[Bal]} Ch.I, Prop. 5.4) and bounded. Therefore, it is
decreasing with a global infimum, say, $C\geq0.$ Clearly, if $C=0$ we have
nothing to show. Assume $C>0.$ For each point $g_{2}\left(  t\right)  $ on
$\mathrm{Im}g_{2},$ denote by $g_{1}\left(  s\left(  t\right)  \right)  $ the
unique point on $\mathrm{Im}g_{1}$ realizing the distance $d\left(
g_{2}(t),\mathrm{Im}g_{1}\right)  .$ By \cite{[Bal]} Ch.I, Cor. 5.6, the
function $t\rightarrow d\left(  g_{1}\left(  s\left(  t\right)  \right)
,g_{2}\left(  t\right)  \right)  ,$ $t\geq0$ is convex and by assumption it
decreases to $0.$ As $d\left(  g_{1}\left(  t\right)  ,g_{2}\left(  t\right)
\right)  \searrow C$ it follows that
\begin{equation}
d\left(  g_{1}\left(  s\left(  t\right)  \right)  ,g_{1}\left(  t\right)
\right)  \rightarrow C\mathrm{\ as\ }t\rightarrow \infty. \label{ctria}%
\end{equation}
Since $g_{1}$ is a geodesic, $\left \vert t-s\left(  t\right)  \right \vert
\rightarrow C$ as $t\rightarrow \infty.$ There exists a sequence $t_{n}%
\rightarrow \infty$ such that
\begin{equation}
t_{n}-s\left(  t_{n}\right)  \rightarrow \delta C\mathrm{\ with\ }\left \vert
\delta \right \vert =1. \label{cena}%
\end{equation}
Define the geodesic $\overline{g}_{1}$ by $\overline{g}_{1}(t):=g_{1}(t-\delta
C)$ and we will show that
\[
d\left(  \overline{g}_{1}\left(  t_{n}\right)  ,g_{2}\left(  t_{n}\right)
\right)  \rightarrow0\mathrm{\ as\ }n\rightarrow \infty.
\]
By the triangle inequality we have
\begin{align*}
d\left(  \overline{g}_{1}\left(  t_{n}\right)  ,g_{2}\left(  t_{n}\right)
\right)   &  \leq d\left(  \overline{g}_{1}\left(  t_{n}\right)  ,g_{1}\left(
s\left(  t_{n}\right)  \right)  \right)  +d\left(  g_{1}\left(  s\left(
t_{n}\right)  \right)  ,g_{2}\left(  t_{n}\right)  \right) \\
&  =d\left(  g_{1}\left(  t_{n}-\delta C\right)  ,g_{1}\left(  s\left(
t_{n}\right)  \right)  \right)  +d\left(  g_{1}\left(  s\left(  t_{n}\right)
\right)  ,g_{2}\left(  t_{n}\right)  \right) \\
&  =\left \vert t_{n}-\delta C-s\left(  t_{n}\right)  \right \vert +d\left(
g_{1}\left(  s\left(  t_{n}\right)  \right)  ,g_{2}\left(  t_{n}\right)
\right) \\
&  \rightarrow0+0
\end{align*}
where the first equality holds by definition of $\overline{g}_{1},$ the second
equality holds since $g_{1}$ is a geodesic and the limit is obtained by
\ref{cena} and the fact that $d\left(  g_{1}\left(  s\left(  t\right)
\right)  ,g_{2}\left(  t\right)  \right)  ,$ $t\geq0$ decreases to $0.$
\end{proof}

We will need a Gauss-Bonnet formula stated for a simply connected
non-positively curved surface $P$ with one piece-wise geodesic boundary
component $\partial P$ and finitely many conical points in its interior and/or
its boundary. For each conical point $s\in \partial P$ denote by $\theta
_{P}\left(  s\right)  $ the cone angle at $s$ inside $P.$ Then the following
holds (see Proposition 8 in \cite{[CP]})
\begin{equation}
2\pi \leq%
{\displaystyle \sum \limits_{s\in P\setminus \partial P}}
\left(  2\pi-\theta \left(  s\right)  \right)  +%
{\displaystyle \sum \limits_{s\in \partial P}}
\left(  \pi-\theta_{P}\left(  s\right)  \right)  . \label{gauss_bonnet}%
\end{equation}

\begin{proof}
[Proof of Proposition \ref{condU}]If $\mathrm{Im}g_{1}\cap \mathrm{Im}g_{2}%
\neq \emptyset$ then there must exist a $K\in \mathbb{R}$ such that
\begin{equation}
\mathrm{Im}g_{1}|_{\left[  K,+\infty \right)  }\subset \mathrm{Im}g_{2}
\label{1in2}%
\end{equation}
otherwise uniqueness of geodesic rays would be violated. Clearly, if property
(\ref{1in2}) holds the result follows trivially, so we may assume that
$\mathrm{Im}g_{1}\cap \mathrm{Im}g_{2} = \emptyset.$ If $d\left(
\mathrm{Im}g_{1},\mathrm{Im}g_{2}\right)  =0,$ the result follows from Lemma
\ref{karaneww}. We assume
\begin{equation}
d\left(  \mathrm{Im}g_{1},\mathrm{Im}g_{2}\right)  =C>0. \label{Cdisatnce}%
\end{equation}
and we will reach a contradiction. Consider the convex region $P$ bounded by
\[
\mathrm{Im}g_{1}|_{\left[  0,+\infty \right)  }\cup \mathrm{Im}g_{2}|_{\left[
0,+\infty \right)  }\cup \left[  g_{1}\left(  0\right)  ,g_{2}\left(  0\right)
\right]  \equiv \partial P.
\]
We claim that there exist finitely many conical points in the interior of $P.$
Assume, on the contrary, that there exist infinitely many conical points in
the interior of $P.$ Then, for any positive integer $N$ we may find $T_{N}%
\in \left(  0,+\infty \right)  $ so that the bounded convex region $P_{N}$
bounded by
\[
\mathrm{Im}g_{1}|_{\left[  0,T_{N}\right]  }\cup \mathrm{Im}g_{2}|_{\left[
0,T_{N}\right]  }\cup \left[  g_{1}\left(  0\right)  ,g_{2}\left(  0\right)
\right]  \cup \left[  g_{1}\left(  T_{N}\right)  ,g_{2}\left(  T_{N}\right)
\right]  \equiv \partial P_{T_{N}}%
\]
contains at least $N$ conical points in its interior. This can be done because
$g_{1},g_{2}$ are asymptotic and thus the distance between the segments
$\left[  g_{1}\left(  0\right)  ,g_{2}\left(  0\right)  \right]  $, $\left[
g_{1}\left(  T_{N}\right)  ,g_{2}\left(  T_{N}\right)  \right]  $ tends to
$+\infty$ as $T_{N}\rightarrow \infty$ which implies that $P=\cup
_{N\in \mathbb{N}}P_{T_{N}}.$\newline Since there are finitely many conical
points in $S,\ $there exists $\theta_{0}>0$ such that $2\pi-\theta \left(
s\right)  <-\theta_{0}<0$ for all conical points $s$ in $S$ and, hence,
\[
2\pi-\theta \left(  \widetilde{s}\right)  <-\theta_{0}%
<0\mathrm{\ for\ all\ conical\ points\ }\widetilde{s}\mathrm{\ in\ }%
\widetilde{S}.
\]
In particular, all terms $\left(  2\pi-\theta \left(  \widetilde{s}\right)
\right)  $ in the first summand on the RHS of formula (\ref{gauss_bonnet}) for
$P_{T_{N}}$ are negative and bounded by $-\theta_{0}.$ The terms $\left(
\pi-\theta_{P}\left(  \widetilde{s}\right)  \right)  $ are non-positive for
all $\widetilde{s}\neq g_{1}\left(  0\right)  ,g_{2}\left(  0\right)
,g_{1}\left(  T_{N}\right)  ,g_{2}\left(  T_{N}\right)  .$ It follows that for
$N$ large enough, say $N>3\frac{2\pi}{\theta_{0}},$ the RHS of formula
(\ref{gauss_bonnet}) is negative, a contradiction. This shows that, in fact,
there exist finitely many conical points in the interior of $P.$ Let $M_{1}$
be a bound on the distance between conical points in $P$ from $\left[
g_{1}\left(  0\right)  ,g_{2}\left(  0\right)  \right]  $ and
\[
M_{2}={\displaystyle \sup \limits_{t}}\left \{  d\bigl(g_{1}\left(  t\right)
,g_{2}\left(  t\right)  \bigr)\bigm \vert t\in \left[  0,+\infty \right)
\right \}  .
\]
Then the geodesic segment $\left[  g_{1}\left(  2M_{1}+2M_{2}\right)
,g_{2}\left(  2M_{1}+2M_{2}\right)  \right]  $ splits $P$ into two
subsurfaces; a bounded one containing all conical points of $P$ and an
unbounded one not containing conical points. Hence, up to re-parametrization,
we may assume that $P\setminus \partial P$ does not contain any conical points.
Moreover, we may assume that
\begin{equation}
p\left(  g_{1}\left(  0\right)  \right)  , p\left(  g_{2}\left(  0\right)
\right)  \in S_{1}, \label{segproj}%
\end{equation}
Clearly, this assumption can be made if both $p\left( \mathrm{Im}g_{1} \right)$ 
and $p\left( \mathrm{Im}g_{2} \right)$ intersect $S_{1}$ (by appropriately 
restricting  $g_1$ and $ g_2 $). We need to verify or exclude the following
four additional cases:

(i) if both $p\left( \mathrm{Im}g_{1} \right)$ and $p\left( \mathrm{Im}g_{2} \right)$
are contained in $M_2$ then the desired result follows as $M_2$ 
has strictly negative curvature. 

(ii) if $p\left( \mathrm{Im}g_{1} \right) \subset M_2$ and 
  $p\left( \mathrm{Im}g_{2} \right) \subset S_1$ (or, vice-versa) 
  then $g_1$ and $g_2$ cannot be asymptotic.
  
(iii) if $p\left( \mathrm{Im}g_{1} \right) \subset M_2$ and
    $p\left( g_{2} \right) $ intersects $\mathrm{Im}c$ finitely many times then
     by appropriate restriction of $g_2$ this case is reduced to either case (i)
     or (ii).
     
(iv)     if $p\left( \mathrm{Im}g_{1} \right) \subset M_2$ and
    $p\left( g_{2} \right) $ intersects $\mathrm{Im}c$ infinitely many times,
    we may pick a lift $\widetilde{c}$ of $c$ such that 
    $d(g_2 (0), \mathrm{Im}\widetilde{c}) > d(g_2 (0),g_1 (0)). $ Then  
    $\mathrm{Im}\widetilde{c}$ splits $\widetilde{X}$ into two subsurfaces
and $\partial\widetilde{X}\setminus \left\{\widetilde{c}(-\infty), 
           \widetilde{c}(\infty)  \right\}$ consists of two components 
           one containing $g_1 (\infty)$ and the other $g_2 (\infty) .$
           This is a contradiction since $g_1 ,g_2$ are assumed asymptotic.\\[3mm]
Enumerate the conical points on 
$\mathrm{Im}g_{1}|_{\left[  0,+\infty \right) }$
by
\[
\widetilde{s}_{0}^{1}=g_{1}\left(  0\right)  ,\widetilde{s}_{1}^{1}%
,\widetilde{s}_{2}^{1},\widetilde{s}_{3}^{1},\ldots,\widetilde{s}_{j}%
^{1},\ldots
\]
according to their distance from $g_{1}\left(  0\right)  $, that is,
$\widetilde{s}_{j}^{1}\in \left[  \widetilde{s}_{0}^{1},\widetilde
{s}_{j^{\prime}}^{1}\right]  $ for all $j<j^{\prime}.$ We exclude from the
enumeration any conical point whose angle inside $P$ is $=\pi.$ We also allow
the case where $\mathrm{Im}g_{1}|_{\left[  0,+\infty \right)  }$ contains
finitely many conical points, that is, the above sequence being finite.
Similarly for $\mathrm{Im}g_{2}|_{\left[  0,+\infty \right)  }.$ As the angle
at each $\widetilde{s}_{j}^{1}$ is $>\pi$ we may extend each geodesic segment
$\left[  \widetilde{s}_{j+1}^{1},\widetilde{s}_{j}^{1}\right]  ,$
$j=1,2,\ldots$ to a geodesic segment $\left[  \widetilde{s}_{j+1}^{1}%
,x_{j}^{1}\right]  \ni \widetilde{s}_{j}^{1}$ so that $x_{j}^{1}\in \partial P$
and the angle at $\widetilde{s}_{j}^{1}$ inside $P$ is $=\pi.$ We claim that
$x_{j}^{1}$ cannot belong to $\mathrm{Im}g_{2}|_{\left[  0,+\infty \right)  }$
and, thus, $x_{j}^{1}\in$ $\left[  g_{1}\left(  0\right)  ,g_{2}\left(
0\right)  \right]  .$ Assume on the contrary that $x_{j}^{1}$ for some $j$
belongs to $\mathrm{Im}g_{2}|_{\left[  0,+\infty \right)  }.$ Denote by
$\left[  x_{j}^{1},\xi \right)  _{g_{2}}$ the geodesic sub-ray of $g_{2}$
emanating from $x_{j}^{1}.$ Similarly, denote by $\left[  \widetilde{s}%
_{j}^{1},\xi \right)  _{g_{1}}$ the geodesic sub-ray of $g_{1}$ emanating from
$\widetilde{s}_{j}^{1}.$ Then, the union
\[
\left[  x_{j}^{1},\widetilde{s}_{j}^{1}\right]  \cup \left[  \widetilde{s}%
_{j}^{1},\xi \right)  _{g_{1}}%
\]
is also a geodesic ray from $x_{j}^{1}$ to $\xi$ because the angle at
$\widetilde{s}_{j}^{1}$ is $\geq \pi$ on both sides. This contradicts
uniqueness of geodesic rays in the CAT(0) space $\widetilde{X}.$ For any
$j<j^{\prime}$ the geodesic segments $\left[  \widetilde{s}_{j}^{1},x_{j}%
^{1}\right]  $ and $\left[  \widetilde{s}_{j^{\prime}}^{1},x_{j^{\prime}}%
^{1}\right]  $ cannot intersect, otherwise there would exist two geodesic
segments joining the intersection point with $\widetilde{s}_{j^{\prime}}^{1}$
contradicting uniqueness of geodesic segments. Hence, $x_{j}^{1}$ and
$x_{j^{\prime}}^{1}$ are distinct. Moreover, as the geodesic triangle formed
by $\widetilde{s}_{j^{\prime}}^{1},$ $g_{1}(0)$ and $x_{j^{\prime}}^{1}$
contains the segment $\left[  \widetilde{s}_{j}^{1},x_{j}^{1}\right]  ,$ it
follows that $d\left(  g_{1}\left(  0\right)  ,x_{j}^{1}\right)  <d\left(
g_{1}\left(  0\right)  ,x_{j^{\prime}}^{1}\right)  .$ We denote the latter
property by the symbol $\prec$ and we have shown that
\begin{equation}
x_{j}^{1}\prec x_{j^{\prime}}^{1}\mathrm{\ for\ all\ }j<j^{\prime}.
\label{distinctxj}%
\end{equation}
Do the same with the segments $\left[  \widetilde{s}_{j+1}^{2},\widetilde
{s}_{j}^{2}\right]  ,$ $j=1,2,\ldots$ to obtain points
\[
\left \{  x_{j}^{2}\bigm \vert j=1,2,\ldots \right \}  \subset \left[
g_{1}\left(  0\right)  ,g_{2}\left(  0\right)  \right]
\]
satisfying
\[
x_{j^{\prime}}^{2}\prec x_{j}^{2}\mathrm{\ for\ all\ }j<j^{\prime}.
\]
As above, for any $j,j^{\prime}$ the geodesic segments $\left[  \widetilde
{s}_{j+1}^{1},x_{j}^{1}\right]  $ and $\left[  \widetilde{s}_{j^{\prime}%
+1}^{2},x_{j^{\prime}}^{2}\right]  $ cannot intersect, otherwise there would
exist two geodesic rays joining the intersection point with $\xi$
contradicting uniqueness of geodesic rays. Therefore,
\begin{equation}
x_{j}^{1}\prec x_{j^{\prime}}^{2}\mathrm{\ for\ all\ }j,j^{\prime}.
\label{distinctjjtonos}%
\end{equation}
Let $x^{1}$ (resp. $x^{2}$) be the unique accumulation point of the set
$\left \{  x_{j}^{1}|j=1,2,\ldots \right \}  $ (resp.$\left \{  x_{j}^{2}|j=
1,2,\ldots \right \}  )$. In the case $\mathrm{Im}g_{1}|_{\left[  0,+\infty
\right)  }$ contains finitely many conical points, $x^{1}$ is simply $\max
_{j}\left \{  x_{j}^{1}\right \}  $ and similarly for $x^{2}.$ Moreover, by
(\ref{distinctjjtonos}),
\begin{equation}
x_{j}^{1}\prec x^{1}\mathrm{\ and\ }x^{2}\prec x_{j}^{2}\mathrm{\ for\ all\ }%
j.
\end{equation}

\noindent \textbf{Case A}: $x^{1}=x^{2}.$\\[2mm]To reach a contradiction, pick
points $x_{j}^{1}$ and $x_{j^{\prime}}^{2}$ such that
\[
d\left(  x_{j}^{1},x_{j^{\prime}}^{2}\right)  <C.
\]
Denote by $r_{j}$ (resp. $r_{j^{\prime}}$) the geodesic ray emanating from
$x_{j}^{1}$ (resp. $x_{j^{\prime}}^{2}$) with $r_{j}\left(  +\infty \right)
=\xi$ (resp. $r_{j^{\prime}}\left(  +\infty \right)  =\xi$). By construction,
$r_{j}$ and $g_{1}$ have a common subray and so do $r_{j^{\prime}}$ and
$g_{2}.$ It follows from (\ref{Cdisatnce}) that for large enough $T,$
$d\left(  r_{j}\left(  T\right)  ,r_{j^{\prime}}\left(  T\right)  \right)
\geq C$ and, by choice, $d\left(  r_{j}\left(  0\right)  ,r_{j^{\prime}%
}\left(  0\right)  \right)  <C.$ Therefore, the distance function $t\rightarrow
d\left(  r_{j}\left(  t\right)  ,r_{j^{\prime}}\left(  t\right)  \right)  ,$
$t\geq0$ is not decreasing. This is a contradiction because the distance
function is convex (see \cite{[Bal]} Ch.I, Prop. 5.4) and, as $r_{j}\left(
+\infty \right)  =r_{j^{\prime}}\left(  +\infty \right)  ,$ it is also
bounded.\\[3mm]

\noindent \textbf{Case B}: $x^{1}\neq x^{2}.$ \\[2mm]In this case it is easily
seen that
\begin{equation}
d\left(  x^{1},x^{2}\right)  \geq C. \label{lowebound}%
\end{equation}
For, if $d\left(  x^{1},x^{2}\right)  <C$ we may find points $x_{j}^{1}$ and
$x_{j^{\prime}}^{2}$ such that
\[
d\left(  x_{j}^{1},x_{j^{\prime}}^{2}\right)  <C.
\]
and proceed to reach a contradiction as above. \newline Moreover, it can be
seen that
\begin{equation}
d\left(  \operatorname{Im}r_{1},\operatorname{Im}r_{2}\right)  \geq C.
\label{loweboundnew}%
\end{equation}
where $\operatorname{Im}r_{i},i=1,2$ is the geodesic ray from $x^{i}$ to
$\xi.$\newline To check this, assume $d\left(  \operatorname{Im}%
r_{1},\operatorname{Im}r_{2}\right)  \leq C-c_{0}$ for some $c_{0}>0.$ We may
find points $x_{j}^{1},x_{j^{\prime}}^{2}$ such that
\[
d\left(  x_{j}^{1},x^{1}\right)  =
d\left(  x^{2},x_{j^{\prime}}^{2}\right)  =c_{0}/3.
\]
and, thus, by convexity
\[
d\left(  \operatorname{Im}r_{j},\operatorname{Im}r_{1}\right)  \leq
c_{0}/3\mathrm{\ and\ }d\left(  \operatorname{Im}r_{2},\operatorname{Im}%
r_{j^{\prime}}\right)  \leq c_{0}/3
\]
Since $r_{j}$ (resp. $r_{j^{\prime}}$) and $g_{1}$ (resp. $g_{2}$) have a
common subray
\[
d\left(  \operatorname{Im}g_{1},\operatorname{Im}g_{2}\right)  =d\left(
\operatorname{Im}r_{j},\operatorname{Im}r_{j^{\prime}}\right)  .
\]
It follows that
\[%
\begin{array}
[c]{rcl}%
C & = & d\left(  \operatorname{Im}g_{1},\operatorname{Im}g_{2}\right)
\;=\;d\left(  \operatorname{Im}r_{j},\operatorname{Im}r_{j^{\prime}}\right) \\
& \leq & d\left(  \operatorname{Im}r_{j},\operatorname{Im}r_{1}\right)
+d\left(  \operatorname{Im}r_{1},\operatorname{Im}r_{2}\right)  +d\left(
\operatorname{Im}r_{2},\operatorname{Im}r_{j^{\prime}}\right) \\
& \leq & c_{0}/3+(C-c_{0})+c_{0}/3\\
& = & C-c_{0}/3.
\end{array}
\]
We distinguish two sub-cases

\textbf{Subcase B1}: $p^{-1}\left(  \operatorname{Im}c_{M}\right)  \cap P$ has
finitely many components.\newline Then, by considering subrays of $g_{1}$ and
$g_{2}$, we may assume that
\[
\mathrm{either\ }p\left(  \operatorname{Im}g_{1}\right)  \cup p\left(
\operatorname{Im}g_{2}\right)  \subset M_{2}\mathrm{\ or,\ }p\left(
\operatorname{Im}g_{1}\right)  \cup p\left(  \operatorname{Im}g_{2}\right)
\subset S_{1}.
\]
In the former case, $\operatorname{Im}g_{1},\operatorname{Im}g_{2}$ are
contained in a subsurface of $\widetilde{X}$ which has strictly negative
curvature. Thus, as they are asymptotic, their distance $d\left(
\mathrm{Im}g_{1},\mathrm{Im}g_{2}\right)  $ must be zero contradicting
(\ref{Cdisatnce}). For the latter case, pick distinct points $y_{1},y_{2}$ in the
interior of the segment $\left[  x^{1},x^{2}\right]  $ and denote by $q_{1}$
(resp. $q_{2}$) the geodesic ray emanating from $y_{1}$ (resp. $y_{2}$) with
$q_{1}\left(  +\infty \right)  =\xi$ (resp. $q_{2}\left(  +\infty \right)  =\xi
$). Clearly, $\operatorname{Im}q_{1},\operatorname{Im}q_{2}$ are disjoint,
thus they are contained in a flat subsurface of $\widetilde{X}\cap \left(
P\setminus \partial P\right)  $ and, being at bounded distance, they are
parallel. Pick a geodesic ray $r$ in the interior of the flat half strip bounded
by $q_{1},q_{2}.$ Clearly, $r\left(  +\infty \right)  =\xi$ and $p\left(
r\right)  $ is contained in the compact $e.s.c.s.$ $S_{1}$ with
\[
d\left(  \operatorname{Im}p\left(  r\right)  ,\left \{  s_{1},...,s_{n}%
\right \}  \right)  >0.
\]
By assumption, $\xi \in \partial \widetilde{X}\setminus F_{h}^{nu}$ which implies
that $r$ cannot be closed. Then the above inequality contradicts Corollary
\ref{mathZC}.

\textbf{Subcase B2}: $p^{-1}\left(  \operatorname{Im}c_{M}\right)  \cap P$ has
infinitely many components.\newline We may assume that all such components are
segments with one endpoint on $\operatorname{Im}g_{1}$ and the other on
$\operatorname{Im}g_{2}.$ In this subcase the convex region bounded by
$\operatorname{Im}r_{1}$ and $\operatorname{Im}r_{2}$ consists of infinitely
many Euclidean and hyperbolic quadrilaterals formed by sub-segments of the
components of $p^{-1}\left(  \operatorname{Im}c_{M}\right)  \cap P$ and
sub-segments of $r_{1},r_{2}.$ To fix notation, let
\[
\left \{  \left[  A_{k},B_{k}\right]  \bigm \vert k=1,2,\ldots \right \}
\]
be an enumeration of the components of $p^{-1}\left(  \operatorname{Im}%
c_{M}\right)  \cap P$ such that, for all $k,$ $A_{k}\in \operatorname{Im}%
r_{1},$ $B_{k}\in \operatorname{Im}r_{2}$ and
\[
d\left(  A_{k},r_{1}\left(  0\right)  \right)  <d\left(  A_{k+1},r_{1}\left(
0\right)  \right)  \mathrm{\ and\ }d\left(  B_{k},r_{2}\left(  0\right)
\right)  <d\left(  B_{k+1},r_{2}\left(  0\right)  \right)  .
\]
Each segment $\left[  A_{k},A_{k+1}\right]  $ (resp. $\left[  B_{k}%
,B_{k+1}\right]  $) has length bounded below by some constant depending on the
geometry of $M$ and $S_{1}.$ In other words,
\begin{equation}
\mathrm{there\ exists\ }C^{\prime}\mathrm{\ such\ that\ }d\left(
A_{k},A_{k+1}\right)  >C^{\prime}\mathrm{\ and\ }d\left(  B_{k},B_{k+1}%
\right)  >C^{\prime}, \label{2lower}%
\end{equation}
Set $A_{0}=r_{1}\left(  0\right)  ,$ $B_{0}=r_{2}\left(  0\right)  $ and
denote by $Q_{k},$ $k=1,2,\ldots$ the quadrilateral formed by the segments
$\left[  A_{k-1},A_{k}\right]  ,$ $\left[  A_{k},B_{k}\right]  ,$ $\left[
B_{k-1},B_{k}\right]  $ and $\left[  A_{k-1},B_{k-1}\right]  .$ We may assume
(cf. (\ref{segproj})) that $Q_{0}$ is a Euclidean quadrilateral and so is
$Q_{k}$ for all $k$ even. Consequently, $Q_{k}$ is a hyperbolic quadrilateral
for all $k$ odd. Clearly,
\[
P=\cup_{k=0}^{\infty}Q_{k}%
\]
and for all $m\neq n,$
\[
Q_{m}\cap Q_{n}=\left \{
\begin{array}
[c]{ll}%
\left[  A_{\min \left \{  m,n\right \}  },B_{\min \left \{  m,n\right \}  }\right]
& \mathrm{if\ }\left \vert m-n\right \vert =1\mathrm{\ and}\\
\emptyset & \mathrm{otherwise.}%
\end{array}
\right.
\]
Denote by $a_{k},\beta_{k},\gamma_{k}$ and $\delta_{k}$ the angles of $Q_{k},$
that is,
\[%
\begin{array}
[c]{rcl}%
a_{k} & = & \measuredangle_{A_{k}}\left(  \left[  A_{k},A_{k+1}\right]
,\left[  A_{k},B_{k}\right]  \right) \\
\beta_{k} & = & \measuredangle_{B_{k}}\left(  \left[  A_{k},B_{k}\right]
,\left[  B_{k},B_{k+1}\right]  \right) \\
\gamma_{k} & = & \measuredangle_{B_{k+1}}\left(  \left[  B_{k},B_{k+1}\right]
,\left[  A_{k+},B_{k+1}\right]  \right) \\
\delta_{k} & = & \measuredangle_{A_{k+1}}\left(  \left[  A_{k},A_{k+1}\right]
,\left[  A_{k+1},B_{k+1}\right]  \right)
\end{array}
\]
We have the following relations%

\[%
\begin{array}
[c]{rcl}%
a_{k+1}+\delta_{k}\; \;= & \pi & =\; \; \beta_{k+1}+\gamma_{k}\\
a_{k}+\beta_{k}+\gamma_{k}+\delta_{k} & = & 2\pi,\mathrm{if\ }%
k\mathrm{\ is\ even}\\
a_{k}+\beta_{k}+\gamma_{k}+\delta_{k} & < & 2\pi,\mathrm{if\ }%
k\mathrm{\ is\ odd}%
\end{array}
\]
It follows that for all $k,$%
\begin{equation}
a_{2k-1}+\beta_{2k-1}\lvertneqq a_{2k}+\beta_{2k}=a_{2k+1}+\beta
_{2k+1}\lvertneqq a_{2k+2}+\beta_{2k+2}. \label{doubleincrease}%
\end{equation}
In particular, the sequence
\begin{equation}
\left \{  a_{k}+\beta_{k}\right \}  _{k\in \mathbb{N}}\mathrm{\ is\ increasing.}
\label{singleincreasing}%
\end{equation}
Denote by $A\left(  Q_{2k+1}\right)  $ the area of the quadrilateral
$Q_{2k+1}.$ \newline \textbf{Claim }: the sequence $\left \{  A\left(
Q_{2k+1}\right)  \right \}  $ is bounded below by some $\Lambda>0.$

Assume, on the contrary, that $\lim_{k_{n}\rightarrow \infty}A\left(
Q_{2k_{n}+1}\right)  =0$ for some subsequence $\left \{  A\left(  Q_{2k_{n}%
+1}\right)  \right \}  $ of $\left \{  A\left(  Q_{2k+1}\right)  \right \}
.$\newline Consider the geodesic segment $\left[  A_{2k_{n}+1},B_{2k_{n}%
+2}\right]  $ which splits $Q_{2k_{n}+1}$ into two triangles, say
$T_{2k_{n}+1}$ and $T_{2k_{n}+1}^{\prime}.$ By assumption, the area of
$T_{2k_{n}+1}^{\prime}$ tends to $0$ as $k_{n}\rightarrow \infty.$ By
(\ref{loweboundnew}) the side $\left[  A_{2k_{n}+1},B_{2k_{n}+1}\right]  $ of
$T_{2k_{n}+1}^{\prime}$ is bounded below by $C$ and, by (\ref{2lower}), the
side $\left[  B_{2k_{n}+1},B_{2k_{n}+2}\right]  $ is bounded below by
$C^{\prime}.$ It follows that $\beta_{2k_{n}+1}\rightarrow0.$ In a similar way
we use the geodesic segment $\left[  A_{2k+2},B_{2k+1}\right]  $ to show that
$\alpha_{2k+1}\rightarrow0$. Thus, $\left \{  a_{2k_{n}+1}+\beta_{k_{n}%
+1}\right \}  \rightarrow0,$ contradicting (\ref{singleincreasing}). This
completes the proof of the Claim.

Note that as $M$ has negative curvature bounded away from $0,$ there exists a
constant $\lambda_{M}$ such that
\[
A\left(  Q_{2k+1}\right)  \leq \lambda_{M}\bigl(2\pi-\left(  a_{2k+1}%
+\beta_{2k+1}+\gamma_{2k+1}+\delta_{2k+1}\right)  \bigr)
\]
Combining this inequality with the Claim, we have%
\begin{equation}%
\begin{array}
[c]{rcl}%
2\pi-\left(  a_{2k+1}+\beta_{2k+1}+\gamma_{2k+1}+\delta_{2k+1}\right)  & \geq
& \frac{\Lambda}{\lambda_{M}}\\
a_{2k+1}+\beta_{2k+1}+\gamma_{2k+1}+\delta_{2k+1} & \leq & 2\pi-\frac{\Lambda
}{\lambda_{M}}\\
a_{2k+1}+\beta_{2k+1}+\left(  \pi-\beta_{2k+2}\right)  +\left(  \pi
-a_{2k+2}\right)  & \leq & 2\pi-\frac{\Lambda}{\lambda_{M}}\\
a_{2k+1}+\beta_{2k+1}+\frac{\Lambda}{\lambda_{M}} & \leq & \alpha_{2k+2}%
+\beta_{2k+2}.
\end{array}
\label{manyineq}%
\end{equation}
By (\ref{doubleincrease}) it follows that $a_{k}+\beta_{k}\rightarrow \infty$
as $k\rightarrow \infty,$ a contradiction.

Therefore, the assumption $d\left(  \mathrm{Im}g_{1},\mathrm{Im}g_{2}\right)
=C>0$ (cf (\ref{Cdisatnce})) leads to a contradiction and the result follows
from Lemma \ref{karaneww}.
\end{proof}

We proceed now to show Condition ($\mathrm{D}$). We will need the following 
\begin{lemma}
\label{existence closed geodesic}Let $\varphi$ be a hyperbolic element of
$\Gamma$ and let $\eta=\varphi(-\infty)$ and $\xi=\varphi(\infty)$ be the
repulsive and attractive points of $\varphi$ in $\partial \widetilde{X}.$ Then
any geodesic line $c$ joining $\eta$ and $\xi$ projects to a closed geodesic
in $X.$
\end{lemma}

\begin{proof}
By Proposition 3.3, p. 31 of \cite{[Bal]}, there is an axis $c_{0}$ of
$\varphi$ in $\widetilde{X}$ which projects to a closed geodesic in $X.$ Let
$c$ be a geodesic line of $\widetilde{X}$ joining the points $\eta,$ $\xi.$
Then, by the Flat Strip Theorem, $c$ and $c_{0}$ are parallel in
$\widetilde{X}$ i.e. they bound a flat strip in $\widetilde{X}.$ Therefore,
$c$ is also an axis of $\varphi$ and thus it projects to a closed geodesic in
$S.$
\end{proof}

\begin{proposition}
\label{newdensity}The set
\[
\left \{  \left(  g\left(  +\infty \right)  ,g\left(  -\infty \right)  \right)
\bigm \vert p\left(  g\right)  \mathrm{\ is\ closed\ and\ unique}\right \}
\]
is dense in $\partial^{2}\widetilde{X}.$
\end{proposition}

\begin{proof}
Let $O\times U$ be open in $\partial^{2}\widetilde{X}$ where $O,U$ are
disjoint intervals in $\partial \widetilde{X}.$ By Proposition \ref{dense3},
there exists a hyperbolic $\phi \in \Gamma$ such that
\[
\left(  \phi \left(  +\infty \right)  ,\phi \left(  -\infty \right)  \right)  \in
O\times U.
\]
If the closed geodesic $\beta$ corresponding to the axis $\left(  \phi \left(
+\infty \right)  ,\phi \left(  -\infty \right)  \right)  $ is unique we have
nothing to show. Suppose $\beta$ is not unique. Then it is contained in a flat
strip, hence, by parallel translation we may assume that $\beta$ contains a
conical point $\widetilde{s}$ with cone angle $\theta \left(  \widetilde
{s}\right)  >2\pi.$ The conical point $\widetilde{s}$ splits $\beta$ into two
geodesic rays, denote them by $r_{2}$ and $r_{3},$ which form an angle $=\pi$
inside the flat strip bounded by $\beta$ and the other angle being equal to
$\theta \left(  \widetilde{s}\right)  -\pi.$ Clearly, $r_{2}\left(
+\infty \right)  =\beta \left(  +\infty \right)  \in O$ and $r_{3}\left(
+\infty \right)  =\beta \left(  -\infty \right)  \in U.$ Let $r_{1},r_{4}$ be
geodesic rays with $r_{1}\left(  0\right)  =r_{4}\left(  0\right)
=\widetilde{s}$ such that

\begin{itemize}
\item $r_{1}\left(  +\infty \right)  \in O\setminus \left \{  \beta \left(
+\infty \right)  \right \}  $ and $r_{4}\left(  +\infty \right)  \in
U\setminus \left \{  \beta \left(  -\infty \right)  \right \}  ,$

\item $r_{1}$ and $r_{4}$ do not intersect the interior of the flat strip
bounded by $\beta$ and

\item $\sphericalangle_{\widetilde{s}}\left(  r_{4},r_{1}\right)  >\pi.$
\end{itemize}

Pick a number $\theta_{0}$ satisfying $0<\theta_{0}<\frac{\theta \left(
\widetilde{s}\right)  -2\pi}{2}.$ If $\sphericalangle_{\widetilde{s}}\left(
r_{1},r_{2}\right)  >\theta_{0}$ we may replace $r_{1}$ by a geodesic ray
$r_{1}^{\prime}$ emanating form $\widetilde{s}$ satisfying the above three
properties and such that, in addition, $\sphericalangle_{\widetilde{s}}\left(
r_{1}^{\prime},r_{2}\right)  <\theta_{0}.$ We similarly replace, if necessary,
$r_{4.}$ Hence, we may assume that the (clockwise) angles formed by $r_{i}$ at
$\widetilde{s}$ satisfy the following relations%
\begin{equation}%
\begin{array}
[c]{r}%
0\leq \sphericalangle_{\widetilde{s}}\left(  r_{1},r_{2}\right)  <\theta_{0},\\
\sphericalangle_{\widetilde{s}}\left(  r_{2},r_{3}\right)  =\pi,\\
0\leq \sphericalangle_{\widetilde{s}}\left(  r_{3},r_{4}\right)  <\theta_{0},\\
\mathrm{\ }\sphericalangle_{\widetilde{s}}\left(  r_{4},r_{1}\right)  >\pi.
\end{array}
\label{angle_inequalities}%
\end{equation}
Observe that equality in any of the above inequalities holds if and only if
the images of the corresponding geodesics rays have a geodesic segment in
common. Let $\left(  r_{1}\left(  +\infty \right)  ,r_{2}\left(  +\infty
\right)  \right)  $ be the (open) interval on the boundary $\partial
\widetilde{X}$ between these two points contained in $O$ and $\left(
r_{3}\left(  +\infty \right)  ,r_{4}\left(  +\infty \right)  \right)  $ the
corresponding interval in $U.$ Clearly, these intervals are disjoint. By
Proposition \ref{dense3} and Lemma \ref{existence closed geodesic} there exist
boundary points
\begin{equation}
\eta \in \left(  r_{1}\left(  +\infty \right)  ,r_{2}\left(  +\infty \right)
\right)  ,\zeta \in \left(  r_{3}\left(  +\infty \right)  ,r_{4}\left(
+\infty \right)  \right)  . \label{ginfinity}%
\end{equation}
such that $\eta=g\left(  +\infty \right)  ,\zeta=g\left(  -\infty \right)  $ for
some closed geodesic $g$ in $G\widetilde{X}.$ Clearly, $\left(  g\left(
+\infty \right)  ,g\left(  -\infty \right)  \right)  \in O\times U$ and we show
that $g$ is a unique (closed) geodesic.

We first show that $\mathrm{Im}g$ must contain $\widetilde{s}.$ $\widetilde
{X}$ is homeomorphic to an open disk and the images of the geodesic rays
$r_{i},i=1,2,3,4$ split $\widetilde{X}$ into four open convex regions denoted
by $P_{12},P_{23},P_{34},P_{41}$ bounded by
\[
\mathrm{Im}r_{1}\cup \mathrm{Im}r_{2}\setminus \mathrm{Im}r_{1}\cap
\mathrm{Im}r_{2},\mathrm{Im}r_{2}\cup \mathrm{Im}r_{3},\mathrm{Im}r_{3}%
\cup \mathrm{Im}r_{4}\setminus \mathrm{Im}r_{3}\cap \mathrm{Im}r_{4}%
,\mathrm{Im}r_{4}\cup \mathrm{Im}r_{1}%
\]
respectively. By its definition (cf equation \ref{ginfinity} above)
$\mathrm{Im}g$ intersects $P_{12}$ and $P_{34}.$ Assume on the contrary that
$\widetilde{s}\notin \mathrm{Im}g.$ Then $\mathrm{Im}g$ must intersect either
$P_{23}$ or, $P_{41}.$ If $\mathrm{Im}g$ intersects $P_{23}$ then
$\mathrm{Im}g$ intersects the boundary lines $r_{2}$ and $r_{3}$ which are
sub-rays of $\beta.$ This contradicts uniqueness of geodesic segments. Assume
now that $\mathrm{Im}g$ intersects $P_{41}.$ Then, $\mathrm{Im}g$ intersects
the boundary lines $r_{1}$ and $r_{4},$ that is, there exists $x=r_{1}\left(
t_{x}\right)  \in \mathrm{Im}g\cap \mathrm{Im}r_{1}$ and $y=r_{4}\left(
t_{y}\right)  \in \mathrm{Im}g\cap \mathrm{Im}r_{4}.$ Since $\sphericalangle
_{\widetilde{s}}\left(  r_{4},r_{1}\right)  >\pi$ (cf \ref{angle_inequalities}
above), $r_{1}|_{\left[  0,t_{x}\right]  }\cup r_{4}|_{\left[  0,t_{y}\right]
}$ is a geodesic segment containing $\widetilde{s}$ with endpoints $x,y.$ As
$x,y\in \mathrm{Im}g$ and $\widetilde{s}\notin \mathrm{Im}g$ the geodesic $g$
provides a geodesic segment with endpoints $x,y$ distinct from $r_{2}%
|_{\left[  0,t_{x}\right]  }\cup r_{3}|_{\left[  0,t_{y}\right]  }.$ This also
contradicts uniqueness of geodesic segments.

To see that $g$ is unique, let $g^{\prime}$ be a (closed) geodesic with
$p\left(  g^{\prime}\right)  $ freely homotopic to $p\left(  g\right)  .$ Then
$g^{\prime}\left(  -\infty \right)  =\eta=g\left(  -\infty \right)  $ and
$g^{\prime}\left(  +\infty \right)  =\zeta=g\left(  +\infty \right)  .$ By the
above argument $\mathrm{Im}g^{\prime}$ also contains $\widetilde{s}$ and by
uniqueness of geodesic rays $g$ and $g^{\prime}$ coincide.
\end{proof}

It is plausible to believe that the techniques used in this section to prove
that the CAT(0) surface $X$ satisfies the assumptions ($\Delta$), (U), (C) and
(D) can be applied for the class of multipolyhedra of piecewise constant
curvature $\chi \leq0 $ (See \cite[\S 3.2]{[Pau]}).\\[3mm]
{\bf Acknowledgments} The authors would like to thank the anonymous referee 
for very helpful comments and suggestions which significantly improved 
this paper.

\end{document}